\documentclass[reqno,11pt]{amsart}
\usepackage[colorlinks=true, linkcolor=blue, citecolor=blue]{hyperref}
    
\usepackage{amssymb}
\usepackage{amsmath, graphicx, rotating}
\usepackage{color}
\usepackage{soul}
\usepackage[dvipsnames]{xcolor}
           
\usepackage{ifthen}
\usepackage{xkeyval}
\usepackage{todonotes}  
\usepackage[T1]{fontenc}
\usepackage{lmodern}
\usepackage[english]{babel}

\usepackage{ upgreek }
\usepackage{stmaryrd}
\SetSymbolFont{stmry}{bold}{U}{stmry}{m}{n}
\usepackage{amsthm}
\usepackage{float}

\usepackage{ bbm }
\usepackage{ stmaryrd }
\usepackage{ mathrsfs }
\usepackage{ frcursive }
\usepackage{ comment }

\usepackage{pgf, tikz}
\usetikzlibrary{shapes}
\usepackage{varioref}
\usepackage{enumitem}

\definecolor{rouge}{rgb}{0.7,0.00,0.00}
\definecolor{vert}{rgb}{0.00,0.5,0.00}
\definecolor{bleu}{rgb}{0.00,0.00,0.8}

\usepackage[textheight=7.61in,textwidth=4.98in]{geometry} 
\newtheorem{theorem}{Theorem}[section]
\newtheorem*{theorem*}{Theorem}
\newtheorem{lemma}[theorem]{Lemma}

\newtheorem{proposition}[theorem]{Proposition}

\labelformat{hypothesis}{\textbf{M\kern-0.1mm#1}}

\newtheorem{condition}{Condition}
\newtheorem{conditionA}{A\kern-0.1mm}
\labelformat{conditionA}{\textbf{A\kern-0.1mm#1}}

\renewcommand\dots{\hbox to 1em{.\hss.\hss.}}

\theoremstyle{definition}

\numberwithin{equation}{section}

\def\geq{\geqslant}
\def\leq{\leqslant}

\DeclareMathOperator{\supp}{supp}

\begin{document}

\title[
Large deviations for products of random matrices]
{
Precise large deviation asymptotics for \\ products of random matrices}

\author{Hui Xiao$^1$}\footnotetext[1]{Universit\'{e} de Bretagne-Sud, LMBA UMR CNRS 6205, Vannes, France.}
\author{Ion Grama$^{1,2}$}\footnotetext[2]{Corresponding author: ion.grama@univ-ubs.fr}
\author{Quansheng Liu$^1$}



\begin{abstract}
Let $(g_{n})_{n\geq 1}$ be a sequence of independent identically distributed $d\times d$ real random matrices
with Lyapunov exponent $\gamma$.
For any starting point $x$ on the unit sphere in $\mathbb R^d$,  
we deal with the norm $ | G_n x | $, 
where $G_{n}:=g_{n} \ldots g_{1}$. 
  The goal of this paper is to establish precise asymptotics for large deviation probabilities
$\mathbb P(\log | G_n x | \geq\nobreak n(q+l))$, 
where $q>\gamma $ is fixed and $l$ is vanishing as $n\to \infty$.
We study  both invertible matrices and positive matrices
and give analogous results for the couple $(X_n^x,\log | G_n x |)$ with target functions,
where $X_n^x= G_n x /| G_n x |$.
As applications we improve previous results on the large deviation principle
for the matrix norm $\|G_n\|$ 
and obtain a precise local limit theorem with large deviations.
\end{abstract}

\date{\today}
\subjclass[2010]{ Primary 60F10, 60B20; Secondary 60J05}
\keywords{Product of random matrices; 
Random walk on the general linear group; Random walk on the semigroup of positive matrices;
spectral gap; 
large deviation; Bahadur-Rao theorem.}
\maketitle


\section{Introduction} 
\subsection{Background and main objectives}
One of the fundamental results in the probability theory is the law of large numbers. 
The large deviation theory describes the rate of convergence in the law of large numbers. 
The most important results in this direction are the Bahadur-Rao and the Petrov precise large deviation asymptotics that we recall below
for independent and identically distributed (i.i.d.)\ real-valued random variables $(X_{i})_{i\geq 1}$. 
Let $S_{n}=\sum_{i=1}^{n}X_{i}.$ 
Denote by $I_{\Lambda}$ the set of real numbers $s\geq 0$ 
such that $\Lambda(s):=\log\mathbb{E}[e^{s X_{1}}]< +\infty$
and by $I_{\Lambda}^\circ$ the interior of $I_{\Lambda}$.
Let $\Lambda^{*}$ be the Frenchel-Legendre transform of $\Lambda$.
Assume that $s\in I_{\Lambda}^\circ$ and $q$ are related by $q=\Lambda'(s)$. 
Set $\sigma^2_s = \Lambda''(s).$
From the results of Bahadur and Rao \cite{BR1960} and Petrov \cite{Petrov65} it follows that
if the law of $X_{1}$ is non-lattice, then the following large deviation asymptotic holds true:
\begin{align}
\label{Petrov theorem 001}
\mathbb{P}(S_{n}\geq n(q+l)) \sim 
\frac{\exp( -n \Lambda^*(q+l))}{s \sigma_s \sqrt{2\pi n}}, 
\ n\to\infty,
\end{align}
where
$
\Lambda^*(q+l)
= \Lambda^{*}(q) + sl + \frac{l^2}{2 \sigma_s^2} + O(l^3) 
$  
and $l$ is a vanishing perturbation as $n\to\infty.$ 
Bahadur and Rao  \cite{BR1960} 
have established the equivalence \eqref{Petrov theorem 001} with $l=0$. 
Petrov 
improved it 
by showing that \eqref{Petrov theorem 001} holds uniformly in
$| l | \leq l_n \to 0$ as $n\to \infty.$
Actually, Petrov's result is also uniform in $q$ 
and  is therefore stronger than Bahadur-Rao's theorem
even with $l=0.$    
The relation \eqref{Petrov theorem 001} with $l=0$ and 
its extension to $|l|\leq l_n\to 0$ have multiple implications 
in various domains of probability and statistics.  
The main goal of the present paper is to establish an equivalence 
similar to \eqref{Petrov theorem 001} 
for products of i.i.d.\  random matrices.

Let $(g_{n})_{n\geq 1}$ be a sequence of i.i.d. $d\times d$ real random matrices
defined on a probability space $(\Omega,\mathcal{F},\mathbb{P})$
with common law $\mu$. 
Denote by $\|\cdot\|$ the operator norm of a matrix and by $| \cdot |$ the Euclidean norm in $\mathbb R^d$.  
Set for brevity 
$G_{n}:= g_{n}\ldots g_{1},$ $n\geq 1$. 
The study of asymptotic behavior of the product $G_n$ attracted much attention, since the fundamental work of
Furstenberg and Kesten \cite{FursKesten1960}, where 
the strong law of large numbers for $\log \|G_n\|$ has been established.
Under additional assumptions, Furstenberg \cite{Furstenberg1963} extended it to $\log |G_n x|$, 
for any starting point $x$ on the unit sphere $\mathbb{S}^{d-1} = \{ x \in \mathbb{R}^d: |x| = 1 \}.$
A number of noteworthy  results in this area can be found in Kesten \cite{Kesten1973}, 
Kingman \cite{Kingman1973},
Le Page \cite{LePage1982}, Guivarc'h and Raugi \cite{GR1985},  Bougerol and Lacroix \cite{Bougerol1985}, 
Goldsheid and Guivarc'h \cite{GG1996}, 
Hennion \cite{Hennion1997},
Furman \cite{Furman2002}, 
Hennion and Herv\'e \cite{Hen-Herve2004},
Guivarc'h \cite{Guivarch2015}, 
Guivarc'h and Le Page \cite{GE2016},
Benoist and Quint \cite{BQ2016, BQ2017} 
to name only a few.

In this paper we are interested in asymptotic behaviour of large deviation probabilities for $\log | G_n x |$
where $x \in \mathbb{S}^{d-1}$.
Set $I_{\mu}=\{s\geq 0: \mathbb{E}(\|g_1\|^{s})<+\infty\}.$  
For $s\in I_\mu$, 
let $ \kappa(s)=\lim_{n\to\infty}\left(\mathbb{E}\| G_n \|^{s}\right)^{\frac{1}{n}}.$ 
Define the convex function $\Lambda(s)=\log\kappa(s)$, $s\in I_\mu$, and consider its Fenchel-Legendre transform
$ 
\Lambda^{\ast}(q)=\sup_{s\in I_{\mu}}\{sq-\Lambda(s)\}, 
$
$q\in\Lambda'(I_{\mu}).$ 
Our first objective is to
establish the following 
Bahadur-Rao type precise large deviation asymptotic: 
\begin{align} 
\label{le page result1}
\mathbb{P}(\log |G_{n} x |\geq nq) \sim
 \bar r_{s}(x) 
\frac{ \exp \left( -n\Lambda^{*}(q)  \right)  }  {s\sigma_s\sqrt{2\pi n}}, \ n\to\infty,
\end{align}
where $\sigma_s>0,$ 
$\bar r_s = \frac{r_s}{ \nu_s( r_s) } >0,$
$r_s$ and $\nu_{s}$ are, respectively, the unique up to a constant
 eigenfunction and unique probability eigenmeasure of the transfer operator
$P_s$ corresponding to the eigenvalue $\kappa(s)$ 
(see Section  \ref{sec:resnorms} for precise statements).
In fact, to enlarge the area of applications in \eqref{le page result1} it is useful to add
a vanishing perturbation for $q$.
In this line we obtain the following Petrov type large deviation expansion: 
under appropriate conditions, uniformly in $| l |\leq l_n \to 0 $ as $n\to\infty,$ 
\begin{align} \label{intro entries 01}
\mathbb{P}(\log | G_n x | \geq n(q+l))
\sim \bar r_{s}(x) 
\frac{ \exp\left(  -n\Lambda^{*}(q+l) \right)} {s\sigma_{s}\sqrt{2\pi n}}, 
\ \ n\to\infty.
\end{align} 
As an consequence of 
\eqref{intro entries 01}  
we are able to infer new results, 
such as large deviation principles for $\log \| G_n \|$, 
see Theorem \ref{Thm-LDP-Norm}. 
From \eqref{intro entries 01} we also deduce a local large deviation asymptotic: 
there exists
a sequence $\Delta_n > 0$ converging to $0$ such that,
uniformly in  $\Delta \in [\Delta_n, o(n)]$, 
\begin{align} \label{local the001}
\mathbb{P}(\log |G_{n} x |\in [nq, nq + \Delta ) ) \sim 
\Delta \frac{\bar r_{s}(x)} 
{ s \sigma_s \sqrt{2\pi n} } e^{-n\Lambda^{*}(q)},\  n\to\infty.
\end{align}

Our results are established for both invertible matrices and positive matrices. 
For invertible matrices, 
Le Page \cite{LePage1982} 
has obtained \eqref{le page result1} 
 for $s>0$ small enough under more restrictive conditions, 
 such as the existence of exponential moments of $\|g_1\|$ and $\|g_1^{-1}\|$. 
The asymptotic \eqref{le page result1} clearly implies a large deviation result due to Buraczewski and Mentemeier 
\cite{BS2016}
which holds for invertible matrices and positive matrices:
for $q=\Lambda'(s)$ and $s\in I_\mu^{\circ}$,  
there exist two constants
$0<c_s<C_s<+\infty$ 
such that
\begin{align} \label{LDbounds001}
c_s\leq \liminf_{n\to\infty}\frac{\mathbb{P}(\log |G_{n} x|\geq nq)}{\frac{1}{\sqrt{n}}~e^{-n\Lambda^{\ast}(q)}}
\leq \limsup_{n\to\infty}\frac{\mathbb{P}(\log |G_{n} x|\geq nq)}{\frac{1}{\sqrt{n}}~e^{-n\Lambda^{\ast}(q)}}
\leq C_s.
\end{align}




Consider the Markov chain $X_{n}^{x}:=G_n x / | G_n x|$.
Our second objective is to 
give precise large deviations for the couple $(X_{n}^{x}, \log | G_n x |)$
with target functions. 
We prove that for any H\"{o}lder continuous target
function $\varphi$ on 
$X_{n}^{x}$, and
any target function $\psi$ on 
$\log |G_{n} x|$
such that $y\mapsto e^{-sy}\psi(y)$ is directly Riemann integrable,
it holds that
\begin{align}  \label{Introdution result2}
&\mathbb{E} \Big[ \varphi(X_{n}^{x}) \psi(\log |G_{n} x| -n(q+l)) \Big] \nonumber \\
&\qquad  \sim 
\bar r_{s}(x)\nu_{s}(\varphi) 
\int_{\mathbb{R}}e^{-sy}\psi(y)dy
\ \frac{ \exp\left(  -n\Lambda^{*}(q+l) \right)} {\sigma_{s}\sqrt{2\pi n}},
\ \ n\to\infty.
\end{align}
As a special case
 of \eqref{Introdution result2} with $l=0$ and $\psi$ compactly supported we obtain 
Theorem 3.3 of Guivarc'h \cite{Guivarch2015}.
With $l=0$, $\psi$ the indicator function of the interval $[0,\infty)$ and $\varphi=r_s$, we get
 the main result in \cite{BS2016}.  

Our third objective is to establish asymptotics for lower large deviation probabilities:
we prove that
for $q=\Lambda'(s)$ with $s<0$ sufficiently close to $0$, 
it holds, 
uniformly in  $|l|\leq l_n$, 
\begin{align} \label{intro s negative001}
\mathbb{P} \big( \log|G_n x|  \leq n(q+l) \big) = 
   \bar r_{s}(x) \frac{ \exp \left( -n \Lambda^*(q+l) \right)  }{ - s \sigma_s \sqrt{ 2 \pi n} } ( 1 + o(1)).
\end{align}
This sharpens the large deviation principle established in \cite[Theorem 6.1]{Bougerol1985} for 
invertible matrices.
Moreover, we extend the large deviation asymptotic \eqref{intro s negative001} 
to the couple $(X_n^x, \log |G_n x|)$ with target functions.



\subsection{Proof outline}
Our proof 
is different from the standard approach of Dembo and Zeitouni \cite{Dembo2009} based on the Edgeworth expansion,
which has been employed for instance in \cite{BS2016}.
In contrast to \cite{BS2016}, 
we start with the identity 
\begin{align} \label{intro001}
\frac{e^{ n \Lambda^{*}(q + l) } }{r_{s}(x)}
& \mathbb{P} \big( \log |G_{n} x| \geq n (q+l) \big) \nonumber \\
 & = e^{ n h_s(l) }  \mathbb{E}_{\mathbb{Q}_{s}^{x}}
\Big( \frac{ \psi_s ( \log |G_{n} x| - n (q+l) ) }{ r_{s}(X_{n}^x) }\Big),
\end{align}
where $\mathbb{Q}_{s}^{x}$ is the change of measure defined in Section \ref{sec:spec gap norm} 
for the norm cocycle $\log |G_n x|$, 
$\psi_s(y)=e^{-sy} \mathbbm{1}_{\{y\geq 0\}}$ and $h_s(l) = \Lambda^{*}(q + l) - \Lambda^{*}(q) -sl$.
Usually the expectation in the right-hand side of \eqref{intro001} is handled via the 
Edgeworth expansion for the distribution function ${\mathbb{Q}_{s}^{x}}\big( \frac { \log |G_n x| - nq}{\sqrt{n} \sigma_s } \leq t \big)$;
however, the presence of the multiplier $r_s(X^x_n)^{-1}$ makes this impossible.
Our idea is to replace the function $\psi_s$ with some 
upper and lower smoothed bounds
using a technique
from 
Grama, Lauvergnat and Le Page \cite{GLE2017}.
For simplicity we deal only with the upper bound 
$\psi_s \leq {\psi}_{s,\varepsilon}^+  * {\rho}_{\varepsilon^2}$,
where $\psi^+_{s,\varepsilon}(y) = \sup_{y': |y'-y| \leq \varepsilon} \psi_s(y')$, for some $\varepsilon >0$,
and $\rho_{\varepsilon^2}$ is a density function on the real line
satisfying the following properties:
the Fourier transform $\widehat{\rho}_{\varepsilon^2}$ is supported on $[-\varepsilon^{-2}, \varepsilon^{-2}]$,
 has a continuous extension in the complex plane and is analytic in the domain 
$\{ z \in \mathbb{C}: |z| < \varepsilon^{-2}, \Im z \neq 0 \}$, see Lemma \ref{LemAnalyExten}.
Let $R_{s,it}$ be the perturbed operator defined by
$R_{s,it}(\varphi)(x)
=\mathbb{E}_{\mathbb{Q}_{s}^{x}} [\varphi(X_1) e^{it ( \log |g_1 x| - q )}],$
for any H\"{o}lder continuous function $\varphi$ on the unit sphere $\mathbb{S}^{d-1}.$
Using the inversion formula we obtain the following upper bound: 
\begin{align} \label{intro002}
& \mathbb{E}_{\mathbb{Q}_{s}^{x}} \Big(\frac{  \psi_s( \log |G_n x| - n (q+l) )  }{r_{s}(X_{n}^x)} \Big) \nonumber \\
&\qquad \qquad\leq  \frac{1}{2\pi} \int_{\mathbb{R}}  e^{-itln}  R^{n}_{s,it}(r_{s}^{-1})(x)
\widehat{\psi}_{s,\varepsilon}^+(t) \widehat{\rho}_{\varepsilon^2} (t)
dt,
\end{align}
where
$R^{n}_{s,it}$ is the $n$-th iteration of $R_{s,it}$. 
The integral in the right-hand side of \eqref{intro002} is decomposed into two parts:
\begin{align}\label{Intro-Integ}
e^{ n h_s(l) } \Big\{ \int_{ |t| < \delta } + \int_{ |t| \geq \delta} \Big\}
e^{-itln}  R^{n}_{s,it}(r_{s}^{-1})(x) \widehat{\psi}_{s,\varepsilon}^+(t) \widehat{\rho}_{\varepsilon^2} (t) dt.
\end{align}  
Since $\widehat{\rho}_{\varepsilon^2}$ is compactly supported on $\mathbb R$ and
$\mu$ is non-arithmetic, the second integral in \eqref{Intro-Integ} decays exponentially fast to $0$. 
To deal with the first integral in \eqref{Intro-Integ}, we make use of 
spectral gap decomposition for the perturbed operator $R_{s,it}$: 
$R^{n}_{s,it} = \lambda^{n}_{s,it} \Pi_{s,it} + N^{n}_{s,it}.$ 
Taking into account the fact that the remainder term $N^{n}_{s,it}$ decays exponentially fast to $0$,
the main difficulty is to investigate the integral:  
\begin{align*}
e^{ n h_s(l) } \int_{ -\delta }^{\delta} e^{-itln}  
\lambda^{n}_{s,it} \Pi_{s,it} (r_{s}^{-1})(x) \widehat{\psi}_{s,\varepsilon}^+(t) \widehat{\rho}_{\varepsilon^2} (t) dt.
\end{align*}
To find the exact asymptotic of this integral, we can apply the saddle point method
(see Fedoryuk \cite{Fedoryuk1987}).
This is possible, since by the analyticity of the functions $\widehat{\psi}_{s,\varepsilon}^+$ and $\widehat{\rho}_{\varepsilon^2}$,
 one can apply Cauchy's integral theorem  to change the integration path so that it passes through the saddle point  
$z_0 = z_0(l)$, which is the unique solution of the saddle point equation $\log \lambda_{s,z} = zl$.

%
%
%

The lower bound of the integral in \eqref{intro001} is a little more delicate, but can be treated in a similar way.
The passage to the targeted version is done by using approximation techniques.

We end this section by fixing some notation, 
which will be used throughout the paper. 
We denote by $c$, $C$, eventually supplied with indices, absolute constants whose values may change from line to line.
By $c_\alpha$, $C_{\alpha}$ we mean constants depending only on the index $\alpha.$
The interior of a set $A$ is denoted by $A^\circ$. 
Let $\mathbb N = \{1,2,\ldots\}$.
For any integrable function $\psi: \mathbb{R} \to \mathbb{C}$, define its Fourier transform by
$\widehat{\psi} (t) = \int_{\mathbb{R}} e^{-ity} \psi(y) dy$, $t \in \mathbb{R}$. 
For a matrix $g$, its transpose is denoted by $g^{\mathrm{T}}.$
For a measure $\nu$ and a function $\varphi$ we write $\nu(\varphi)=\int \varphi d\nu.$

\section{Main results}\label{sec.prelim}
\subsection{Notation and conditions}\label{subsec.notations}
The space $\mathbb{R}^d$ is equipped  with the standard scalar product $\langle \cdot, \cdot\rangle$ 
and the Euclidean norm $|\cdot|$. 
For $d\geq 1$, let $M(d,\mathbb{R})$ be the set of $d\times d$ matrices  with entries in $\mathbb R$
equipped with the operator norm $\|g\|=\sup_{x\in \mathbb{S}^{d-1}}|g x|$, for $g \in M(d,\mathbb{R})$, 
where $\mathbb{S}^{d-1}=\{x\in \mathbb{R}^{d}, |x|=1\}$ is the unit sphere.

We shall work with products of invertible or positive matrices (all over the paper we use the term positive in the wide sense, i.e.\ each entry is non-negative).
Denote by $\mathscr G=GL(d,\mathbb R)$ the general linear group of invertible matrices of $M(d,\mathbb R).$
A positive matrix $g\in M(d,\mathbb R)$ is said to be \emph{allowable},
if every row and every column of $g$ has a strictly positive entry.
Denote by $\mathscr G_+$ the multiplicative semigroup of  allowable positive   matrices of $M(d,\mathbb R)$.
We write $\mathscr G_+^\circ $ for the subsemigroup of $\mathscr G_+$ with strictly positive entries.

Denote by $\mathbb{S}^{d-1}_{+}=\{x\geq 0 : |x|=1\}$ the intersection of the unit sphere
with the positive quadrant.
To unify the exposition, 
we use the symbol $\mathcal{S}$ to denote
$\mathbb{S}^{d-1}$ in the case of invertible matrices, 
and $\mathbb{S}^{d-1}_{+}$ in the case of positive matrices.
The space $\mathcal S$ is equipped with the metric $\mathbf d$ which we proceed to introduce.
For invertible matrices, the distance $\mathbf{d}$ is defined as the angular distance (see \cite{GE2016}), i.e.,
for any $x, y \in \mathbb{S}^{d-1}$, $\mathbf{d}(x,y)= |\sin \theta(x,y)|$, 
where $\theta(x,y)$ is the angle between $x$ and $y$. 
For positive matrices, the distance $\mathbf{d}$ is the Hilbert cross-ratio metric 
(see \cite{Hennion1997}) defined by 
$\mathbf{d}(x,y) = \frac{1- m(x,y)m(y,x)}{1 + m(x,y)m(y,x)}$, where 
$m(x,y)=\sup\{\lambda>0 : \   \lambda y_i\leq x_i,\  \forall i=1,\ldots, d    \}$, 
for any two vectors $x=(x_1, \ldots, x_d)$ and $y=(y_1, \ldots, y_d)$ in  $\mathbb{S}_{+}^{d-1}$.

Let $\mathcal{C}(\mathcal{S})$ be the space of continuous functions on $\mathcal{S}$. 
We write $\mathbf{1}$ for the identity function $1(x)$, $x \in \mathcal{S}$. 
Throughout this paper, let $\gamma>0$ be a fixed small constant.  
For any $\varphi\in \mathcal{C(S)}$, set
\begin{align}
\|\varphi\|_{\infty}:= \sup_{x\in \mathcal{S}}|\varphi(x)| \quad \mbox{and} \quad
\|\varphi\|_{\gamma}:= \|\varphi\|_{\infty}
+ \sup_{x,y\in \mathcal{S}}\frac{|\varphi(x)-\varphi(y)|}{\mathbf{d}(x,y)^{\gamma}}, \nonumber
\end{align}
 and introduce the Banach space
$
\mathcal{B}_{\gamma}:=\{\varphi\in \mathcal{C(S)}: \|\varphi\|_{\gamma}<+\infty\}. 
$


For $g\in M(d,\mathbb R)$ and $x \in \mathcal{S}$, 
write $g \cdot x= \frac{gx}{|gx|}$ for the projective action of $g$ on $\mathcal{S}$.
For any $g \in M(d,\mathbb R)$, set $\iota(g):=\inf_{x\in \mathcal S}|gx|.$
For both invertible matrices and allowable  positive  matrices, it holds that $\iota(g)>0.$
Note that for any invertible matrix $g$, we have $\iota(g) = \| g^{-1} \|^{-1}$.

Let $(g_{n})_{n\geq 1}$ 
be a sequence of i.i.d.\  random matrices of the same probability law $\mu$ on $M(d,\mathbb{R})$.
Set $G_n=g_{n}\ldots g_{1},$ for $n\geq 1.$ 
Our goal is to establish, under suitable conditions, 
 a large deviation equivalence similar to \eqref{Petrov theorem 001}
 for the norm cocycle
$\log |G_n x|$ 
for invertible matrices and positive matrices.
In both cases,
we denote by $\Gamma_{\mu}:=[\supp\mu]$ 
the smallest closed semigroup of $M(d,\mathbb{R})$ generated by $\supp \mu$ (the support of $\mu$),
that is, $\Gamma_{\mu} = \overline{\cup_{n=1}^\infty \{ \supp \mu\}^n}$. 

Set 
$$ 
I_{\mu} 
=\{s\geq 0: \mathbb{E}(\|g_1\|^{s})<+\infty\}.
$$
Applying H\"{o}lder's inequality to $\mathbb{E}(\|g_1\|^{s})$, it is easily seen that $I_{\mu}$ is an interval.
We  make use of the following exponential moment condition:

\begin{conditionA}\label{Aexp}
There exist $s\in I_\mu^\circ$  and $\alpha\in(0,1)$ such that
$\mathbb{E} \| g_1 \|^{s+\alpha} \iota(g_{1})^{-\alpha} < + \infty.$
\end{conditionA}

For invertible matrices, we introduce the following strong irreducibility and proximality conditions,  
where we recall that a matrix $g$ is said to be \emph{proximal} if it has an algebraic simple dominant eigenvalue.
 \begin{conditionA}\label{A1}
{\rm (i)(Strong irreducibility)}
No finite union of proper subspaces of $\mathbb{R}^d$ is $\Gamma_{\mu}$-invariant.

{\rm (ii)(Proximality)} $\Gamma_{\mu}$ contains at least one proximal matrix. 
\end{conditionA}
The conditions of strong irreducibility and proximality are always satisfied for $d=1$. 
If $g$ is proximal, denote by $\lambda_{g}$ its dominant eigenvalue 
and by $v_{g}$ the associated normalized eigenvector ($|v_g|=1$).
In fact, $g$ is proximal iff the space $\mathbb{R}^{d}$ can be decomposed as $\mathbb{R}^{d}=\mathbb{R}\lambda_{g}\oplus V'$ 
such that $gV'\subset V'$ 
and the spectral radius of $g$ on the invariant subspace $V'$
is strictly less than $|\lambda_{g}|$.
For invertible matrices, condition \ref{A1}  
implies that the Markov chain $X_{n}^{x}$ has a unique $\mu$-stationary measure, which is supported on
$$
V(\Gamma_{\mu})=\overline{\{ \pm v_{g}\in 
\mathbb S^{d-1}:  g\in\Gamma_{\mu}, \ g \mbox{ is proximal} \}}.
$$

For positive matrices, introduce the following condition: 
\begin{conditionA}\label{A2}
{\rm (i) (Allowability)}
Every $g\in\Gamma_{\mu}$ is allowable.

{\rm (ii) (Positivity)}
$\Gamma_{\mu}$ contains at least one matrix belonging to $\mathscr G_+^\circ$.
\end{conditionA}
It can be shown (see \cite[Lemma 4.3]{BDGM2014}) that for positive matrices, condition \ref{A2} ensures the existence and uniqueness of the invariant measure for the Markov chain $X_{n}^{x}$ 
supported on
$$
V(\Gamma_{\mu})=\overline{\{v_{g}\in 
\mathbb S^{d-1}_+
: g\in \Gamma_{\mu}, \  g \in \mathscr G_+^{\circ} \}}.
$$
In addition, $V(\Gamma_{\mu})$ is the unique minimal $\Gamma_{\mu}$-invariant subset (see \cite[Lemma 4.2]{BDGM2014}).
According to the Perron-Frobenius theorem, a strictly positive matrix always has a unique dominant eigenvalue,
so condition \ref{A2}(ii) implies condition \ref{A1}(ii) for $d>1$.

For any $s\in I_{\mu}$, for invertible matrices and for  positive matrices, 
the following limit exists (see \cite{GE2016} and \cite{BS2016}): 
\begin{align}
\kappa(s)=\lim_{n\to\infty} \left(\mathbb{E} \| G_n \|^{s}\right)^{\frac{1}{n}}.  \nonumber
\end{align}
The function 
$\Lambda=\log\kappa: I_{\mu} \to \mathbb R$ is convex and 
analytic on $I_{\mu}^{\circ}$
(it plays the same role as the $\log$-Laplace transform of $X_1$ in the real i.i.d.\ case). 
Introduce the
Fenchel-Legendre transform of $\Lambda$ by
$\Lambda^{\ast}(q)=\sup_{s\in I_{\mu}}\{sq-\Lambda(s)\},$
$q\in\Lambda'(I_{\mu}).$
We have that $\Lambda^*(q)=s q - \Lambda(s)$ 
if $q=\Lambda'(s)$ for some $s\in I_{\mu}$, 
which implies $\Lambda^*(q)\geq0$ on $\Lambda'(I_{\mu})$ since $\Lambda(0)=0$ and $\Lambda(s)$ is convex on $I_{\mu}$.

We say that the measure $\mu$ 
is \emph{arithmetic}, if there exist $t>0$, $\beta \in[0,2\pi)$ and a function 
$\vartheta: \mathcal{S}
 \to \mathbb{R}$ 
such that for any $g\in \Gamma_{\mu}$ and any $x\in V(\Gamma_{\mu})$, we have
$
\exp[it\log|gx|-i\beta + i\vartheta(g\!\cdot\!x)-i \vartheta(x)]=1. 
$
For positive matrices, we need the following condition:
\begin{conditionA}\label{A3}
{\rm (Non-arithmeticity)}  The measure $\mu$ is non-arithmetic.
\end{conditionA}
A simple sufficient condition established in \cite{Kesten1973} for  
the measure $\mu$
to be non-arithmetic
is that the additive subgroup of $\mathbb{R}$ generated by the set
$\{ \log \lambda_{g} : g\in \Gamma_{\mu}, \  g \in \mathscr G_+^\circ \}$
is dense in $\mathbb{R}$ (see \cite[Lemma 2.7]{BS2016}). 

Note that for positive matrices, 
condition \ref{A3} is used to ensure that 
$\sigma_s^2=\Lambda''(s) >0$.
For invertible matrices, 
condition \ref{A1} 
implies the non-arithmeticity of the  measure $\mu$, 
hence, $\sigma_s$ is also strictly positive 
(for a proof see Guivarc'h and Urban \cite[Proposition 4.6]{GU2005}).

For any $s\in I_{\mu}$, the transfer operator $P_s$ and the conjugate transfer operator $P_{s}^{*}$ 
are defined, for any $\varphi \in \mathcal{C(S)}$ and $x\in \mathcal S$, by
\begin{align}\label{transfoper001}
\! P_{s}\varphi(x) \!=\! \int_{\Gamma_{\mu}}  \!  |g_1 x |^{s} \varphi( g_1\!\cdot\!x ) \mu(dg_1), \ 
P_{s}^{*}\varphi(x) \!=\! \int_{\Gamma_{\mu}} \!  |g_1^{\mathrm{T}}x|^{s} \varphi(g_1^{\mathrm{T}}\!\cdot\!x) \mu(dg_1),
\end{align}
which are bounded linear on $\mathcal{C(S)}$. 
Under condition \ref{A1} for invertible matrices, 
or condition \ref{A2} for positive matrices, 
the operator $P_s$ 
has a unique probability eigenmeasure $\nu_s$ on $\mathcal{S}$
corresponding to the eigenvalue $\kappa(s)$: 
$P_s \nu_s = \kappa(s)\nu_s.$
Similarly, the operator $P_{s}^{*}$ 
has a unique probability eigenmeasure $\nu^*_s$
corresponding to the eigenvalue $\kappa(s)$: 
$P_{s}^{*} \nu^*_s = \kappa(s)\nu^*_s.$
Set, for $x\in \mathcal{S}$, 
$$
r_{s}(x)= \int_{\mathcal{S}} |\langle x, y\rangle|^{s}\nu^*_{s}(dy), \ \
r_{s}^*(x)= \int_{\mathcal{S}} |\langle x, y\rangle|^{s}\nu_{s}(dy).
$$ 
Then, $r_s$ is the unique, up to a scaling constant, 
strictly positive eigenfunction of $P_s$:
$P_s r_s = \kappa(s)r_s$;
similarly
$r^*_s$ is the unique, up to a scaling constant, 
strictly positive eigenfunction of $P_{s}^{*}$: $P_{s}^{*} r^*_s = \kappa(s)r^*_s$. 
We refer for details to Section \ref{sec:spec gap norm}.

Below we shall also make use of normalized eigenfunction $\bar r_s$ defined by 
$\bar r_s(x)= \frac{r_s(x)}{ \nu_s(r_s) }$, $x \in \mathcal{S}$, which is 
strictly positive and H\"{o}lder continuous on the projective space $\mathcal{S}$, 
see Proposition \ref{transfer operator}. 

\subsection{Large deviations for the norm cocycle}  \label{sec:resnorms}

The following theorem gives the exact asymptotic behavior 
of the large deviation probabilities for the norm cocycle. 
\begin{theorem} \label{main theorem1}  
Assume that $\mu$ satisfies either conditions \ref{Aexp}, \ref{A1} for invertible matrices, 
or conditions \ref{Aexp}, \ref{A2}, \ref{A3} for positive matrices.
Let $q=\Lambda'(s)$, where $s\in I_{\mu}^{\circ}$. 
Then for any positive sequence $(l_n)_{n \geq 1}$ satisfying $\lim_{n\to \infty}l_n = 0$, 
we have, as $n \to \infty$, 
uniformly in $x\in \mathcal{S}$ and $|l|\leq l_n$, 
\begin{align}\label{theorem-main001}
\mathbb{P} \big( \log|G_n x|  \geq n(q+l) \big) = 
   \bar r_{s}(x) \frac{ \exp \left( -n \Lambda^*(q+l) \right)  }{s\sigma_s\sqrt{2\pi n}} ( 1 + o(1)).
\end{align}
In particular, with $l=0$, as $n \to \infty$, uniformly in $x\in \mathcal{S}$, 
\begin{align}\label{devrez001}
\mathbb{P} \big( \log|G_n x|  \geq nq \big) =  
  \bar r_{s}(x) \frac{ \exp \left( -n \Lambda^*(q) \right)  }{ s \sigma_s \sqrt{2 \pi n} } ( 1+ o(1)). 
\end{align}
\end{theorem}
The rate function $\Lambda^*(q+l)$ admits the following expansion: for $q=\Lambda'(s)$ and $l$ in a small neighborhood
of $0$, we have
\begin{align} \label{Def Jsl}
\Lambda^*(q+l)
   = \Lambda^{*}(q) + sl + \frac{l^2}{2 \sigma_s^2} - \frac{l^3}{\sigma_s^3} \zeta_s\Big(\frac{l}{\sigma_s}\Big), 
\end{align}
where $\zeta_s(t)$ is the 
Cram\'{e}r series,
$\zeta_s(t) = \sum_{k=3}^{\infty} c_{s,k} t^{k-3}= \frac{ \Lambda'''(s) }{6 \sigma_s^3} + O(t),$
with $\Lambda'''(s)$ and $\sigma_s$ defined in Proposition \ref{perturbation thm}.
We refer for details to Lemma \ref{lemmaCR001},
where the coefficients $c_{s,k}$ are given in terms of the cumulant generating function $\Lambda=\log \kappa$. 

For invertible matrices, a point-wise version of \eqref{devrez001},  
without $\sup_{x\in \mathcal{S}}$ and with $l=0$, namely the asymptotic \eqref{le page result1}, 
has been first established by Le Page \cite[Theorem 8]{LePage1982} 
for small enough $s>0$
under a stronger exponential moment condition.
For  positive matrices, the asymptotic \eqref{devrez001} is new and  implies 
the  large deviation bounds \eqref{LDbounds001}
established in 
Buraczewski and Mentemeier  \cite[Corollary 3.2]{BS2016}.
We note that there is a misprint in \cite{BS2016}, where $e^{nsq}$ should be replaced by $e^{\Lambda^*(q)}$.

Now we consider the precise large deviations for the couple $(X_n^x, \log |G_{n} x|)$
with target functions $\varphi$ and $\psi$ 
on  $X_{n}^{x}:=G_n \!\cdot\! x$ and  $\log |G_{n} x|$, respectively.
\begin{theorem}  \label{main theorem3}
Assume the conditions of Theorem \ref{main theorem1} and
let $q=\Lambda'(s)$ for $s\in I_\mu^\circ$.
Then, for any $\varphi \in \mathcal{B}_{\gamma}$, 
 any measurable function $\psi$ on $\mathbb{R}$ 
such that $y\mapsto e^{-sy}\psi(y)$ is 
directly Riemann integrable,
and any positive sequence $(l_n)_{n \geq 1}$ satisfying $\lim_{n\to \infty}l_n = 0$, 
we have, as $n \to \infty$, uniformly in $x \in \mathcal{S}$ and $|l| \leq l_n$, 
\begin{align} \label{Petrov-Target01}
& \mathbb{E} \Big[ \varphi(X_{n}^{x})\psi(  \log|G_n x|-n(q+l)) \Big]  \nonumber\\
&  \qquad\qquad
  = \bar r_{s}(x)   
  \frac{ \exp \left( -n \Lambda^*(q+l) \right) }{\sigma_s \sqrt{2\pi n} }
  \Big[  \nu_s(\varphi) \int_{\mathbb{R}} e^{-sy} \psi(y) dy  +  o(1)  \Big]. 
\end{align}
\end{theorem}
With
$\varphi = \mathbf{1}$ and $\psi(y)=\mathbbm{1}_{\{y\geq0\}}$ for $y\in \mathbb R,$
we obtain Theorem \ref{main theorem1}.
For invertible matrices and with $l=0$,  
Theorem \ref{main theorem3} 
strengthens  the point-wise large deviation result stated 
in Theorem 3.3 of Guivarc'h \cite{Guivarch2015}, since 
 we do not assume the function $\psi$ to be compactly supported and 
our result is uniform in $x\in \mathcal{S}$. 
By the way we would like to remark that 
in Theorem 3.3 of \cite{Guivarch2015} $\kappa^n(s)$ should be replaced by $\kappa^{-n}(s)$,
and $\nu_s(\varphi r_s^{-1})$ should be replaced by $\frac{\nu_s(\varphi)}{\nu_s(r_s)}$.  
For positive matrices, Theorem \ref{main theorem3}  is new. 
Since $r_s$ is a strictly positive and H\"{o}lder continuous function on $\mathcal{S}$ 
(see Proposition \ref{transfer operator}),
taking $\varphi=r_s$ and $\psi(y)=\mathbbm{1}_{\{y\geq0\}}$, $y\in \mathbb R$ in Theorem \ref{main theorem3},  
we get the main result of \cite{BS2016} (Theorem 3.1).

Unlike the case of i.i.d.\ real-valued random variables,
Theorems \ref{main theorem1} and \ref{main theorem3} 
do not imply the similar asymptotic for lower large deviation probabilities $\mathbb{P}( \log|G_n x|  \leq n(q+l))$, where $q <\Lambda'(0)$.
To formulate our results, we need an exponential moment condition, as in 
Le Page \cite{LePage1982}.
For $g \in \Gamma_{\mu}$, set $N(g) = \max\{ \|g\|, \iota(g)^{-1} \}$, which reduces to  $N(g) = \max\{ \|g\|, \|g^{-1}\| \}$
for invertible matrices.
\begin{conditionA}\label{CondiMoment} 
There exists a constant $\eta \in (0,1)$ such that $\mathbb{E} [N(g_1 )^{\eta}] < +\infty$.
\end{conditionA}

Under condition \ref{CondiMoment}, 
the functions $s \mapsto \kappa(s)$ and $s \mapsto \Lambda(s) = \log \kappa(s)$
can be extended analytically in a small neighborhood of $0$ of the complex plane; 
in this case the expansion \eqref{Def Jsl} still holds and we have $\sigma_s = \Lambda''(s)>0$ for $s<0$ small enough. 
We also need to extend the function $\overline r_s$ for small $s<0$, which is positive and H\"older continuous on the projective space $\mathcal S$, 
as in the case of $s>0$:  we refer to Proposition \ref{transfer operator s negative} 
for details. 

\begin{theorem} \label{Thm-Neg-s}  
Assume that $\mu$ satisfies either conditions \ref{A1}, \ref{CondiMoment} for invertible matrices 
or conditions \ref{A2}, \ref{A3}, \ref{CondiMoment} for positive matrices. 
Then, there exists $\eta_0 < \eta$ such that for any $s \in (-\eta_0, 0)$ and $q=\Lambda'(s)$, 
for any positive sequence $(l_n)_{n \geq 1}$ satisfying $\lim_{n\to \infty}l_n = 0$, 
 we have, as $n \to \infty$, 
uniformly in $x\in \mathcal{S}$ and $|l|\leq l_n$, 
\begin{align*}
\mathbb{P} \big( \log|G_n x|  \leq n(q+l) \big) = 
   \bar r_{s}(x) \frac{ \exp \left( -n \Lambda^*(q+l) \right)  }{ - s \sigma_s \sqrt{ 2 \pi n} } ( 1 + o(1)).
\end{align*}
In particular, with $l=0$, as $n \to \infty$, uniformly in $x\in \mathcal{S}$, 
\begin{align*}
\mathbb{P} \big( \log|G_n x|  \leq nq \big) =  
  \bar r_{s}(x) \frac{ \exp \left( -n \Lambda^*(q) \right)  }{ - s \sigma_s \sqrt{2 \pi n} } ( 1+ o(1)). 
\end{align*}
\end{theorem}

For invertible matrices, this result sharpens the large deviation principle established in \cite{Bougerol1985}.
For positive matrices, our result is new, even for the large deviation principle.

More generally, we also have the precise large deviations result for the couple $(X_n^x, \log |G_n x|)$
with target functions. 

\begin{theorem}  \label{Thm-Neg-s-Target}
Assume the conditions of Theorem \ref{Thm-Neg-s}.    
Then, there exists $\eta_0 < \eta$ such that for any $s \in (-\eta_0, 0)$ and $q=\Lambda'(s)$, 
for any $\varphi \in \mathcal{B}_{\gamma}$, 
 any measurable function $\psi$ on $\mathbb{R}$ 
such that $y\mapsto e^{-sy}\psi(y)$ is 
directly Riemann integrable,
and any positive sequence $(l_n)_{n \geq 1}$ satisfying $\lim_{n\to \infty}l_n = 0$, 
we have, as $n \to \infty$, 
uniformly in $x \in \mathcal{S}$ and $|l| \leq l_n$, 
\begin{align*}
& \mathbb{E} \Big[ \varphi(X_{n}^{x})\psi(  \log|G_n x|-n(q+l)) \Big]  \nonumber\\
&  \qquad\qquad\quad 
  = \bar r_{s}(x)  
  \frac{ \exp \left( -n \Lambda^*(q+l) \right) }{\sigma_s \sqrt{2\pi n} }
  \Big[  \nu_s(\varphi) \int_{\mathbb{R}} e^{-sy} \psi(y) dy  +  o(1)  \Big]. 
\end{align*}
\end{theorem}
With $\varphi = \mathbf{1}$ and $\psi(y)=\mathbbm{1}_{\{y \leq 0 \}}$ for $y\in \mathbb R,$
we obtain Theorem \ref{Thm-Neg-s}.

\subsection{Applications to large deviation principle for the matrix norm}
We use Theorems \ref{main theorem1} and \ref{Thm-Neg-s} 
to deduce large deviation principles for the matrix norm $\|G_n\|$. 
Our first result concerns the upper tail and the second one deals with lower tail. 

\begin{theorem}\label{Thm-LDP-Norm}
Assume the conditions of Theorem \ref{main theorem1}. 
Let $q=\Lambda'(s)$, where $s\in I_\mu^\circ$.
Then, 
for any positive sequence $(l_n)_{n \geq 1}$ with $l_n \to 0$ as $n \to \infty$,
we have, uniformly in $|l|\leq l_n$,
\begin{align*}
\lim_{n\to \infty} 
\frac{1}{n} 
\log \mathbb{P} \big(\log \| G_n \| \geq n(q+l) \big) = -\Lambda^*(q). 
\end{align*}
\end{theorem}
For invertible matrices, with $l=0$, Theorem \ref{Thm-LDP-Norm} improves 
the large deviation bounds  in Benoist and Quint \cite[Theorem 14.19]{BQ2017},
where the authors consider general groups, but without giving the rate function. 
For positive matrices, the result is new for $l=0$ and $l=O(l_n)$. 

\begin{theorem}\label{Thm-LDP-Norm-Negs}
Assume the conditions of Theorem \ref{Thm-Neg-s}. 
Then, there exists $\eta_0 < \eta$ such that for any $s \in (-\eta_0, 0)$ and $q=\Lambda'(s)$, 
for any positive sequence $(l_n)_{n \geq 1}$ with $l_n \to 0$ as $n \to \infty$, 
 we have, uniformly in  $|l|\leq l_n$, 
\begin{align*}
\lim_{n\to \infty} 
\frac{1}{n} 
\log \mathbb{P} \big(\log \| G_n \| \leq n(q+l) \big) = -\Lambda^*(q). 
\end{align*}
\end{theorem}

This result is new for both invertible matrices and positive matrices.

\subsection{Local limit theorems with large deviations}  \label{Applic to LocalLD}
Local limit theorems and 
large and moderate deviations
for sums of i.i.d.\ random variables have been studied by
Gnedenko \cite{gnedenko1948},
Sheep \cite{Sheep1964},
Stone \cite{Stone1965},
Breuillard \cite{Bre2005},
Borovkov and Borovkov \cite{Borovkov2008}.
Moderate deviation results in the local limit theorem for products of invertible random matrices 
have been obtained in \cite[Theorems 17.9 and 17.10]{BQ2017}. 

Taking $\varphi = \mathbf{1}$ and $\psi = \mathbbm{1}_{[a,a+\Delta]},$ where $a \in \mathbb{R}$ and $\Delta >0$
do not depend on $n$, 
it is easy to understand that Theorem \ref{main theorem3}
becomes, in fact, a statement on large deviations in the local limit theorem. 
It turns out that with the Petrov type extension \eqref{Petrov-Target01}
we can derive  the following more general statement where $\Delta$ can increase with $n.$

\begin{theorem}\label{Theorem local LD001}
Assume conditions of Theorem \ref{main theorem1} and
let $q=\Lambda'(s)$. 
Then there exists a sequence $\Delta_n > 0$ 
converging to $0$ 
as $n\to\infty$ 
such that, for any $\varphi \in \mathcal{B}_{\gamma}$, 
for any positive sequence $(l_n)_{n \geq 1}$ with $l_n \to 0$ as $n \to \infty$ and any fixed $a\in \mathbb{R}$,
we have, as $n\to\infty,$
 uniformly in  $\Delta \in [\Delta_n, o(n)]$, $x \in \mathcal{S}$ and 
  $|l| \leq l_n$,  
\begin{align*}
& \mathbb{E} \Big[ \varphi(X_n^x) \mathbbm{1}_{ \{ \log |G_{n} x | \in n(q+l) + [ a, a + \Delta) \} }  \Big]  
   \nonumber\\
& \qquad\qquad  
    = \bar r_{s}(x) e^{-sa} \big( 1 - e^{ -s\Delta } \big) 
\frac{ \exp ( - n \Lambda^{*}(q+l) ) }{ s \sigma_s \sqrt{2\pi n}  } 
   \Big[ \nu_s(\varphi) +  o(1) \Big]. 
\end{align*}
Taking $\varphi = \mathbf{1}$, as $n\to\infty,$
 uniformly in  $\Delta \in [\Delta_n, o(n)]$, $x \in \mathcal{S}$ and $|l| \leq l_n$, 
\begin{align*}
& \mathbb{P} \big(\log |G_{n} x | \in n(q+l) +[a,a+\Delta) \big) \\
& \qquad\qquad   =    \bar r_{s}(x) e^{-sa} \big( 1 - e^{ -s\Delta } \big) 
\frac{ \exp ( - n \Lambda^{*}(q+l) ) }{ s \sigma_s \sqrt{2\pi n}  } \Big[ 1 + o(1) \Big]. 
\end{align*}
\end{theorem}

We can compare this result with Theorem 3.3 in \cite{Guivarch2015},
from which the above equivalence can be deduced for $l=0$ and $\Delta$ fixed.  

It is easy to see that, under additional assumption \ref{CondiMoment}, the assertion of Theorem  \ref{Theorem local LD001}
remains true for  $s<0$ small enough. 
This can be deduced from  Theorem \ref{Thm-Neg-s-Target}: the details are left to the reader.

%


\section{Spectral gap theory for the norm} \label{sec:spec gap norm}
\subsection{Properties of the transfer operator}\label{subsec a change of measure}

Recall that the transfer operator $P_{s}$ and the conjugate operator
$P_{s}^{*}$ are defined by
\eqref{transfoper001}. 
Below $P_s\nu_{s}$ stands for the measure on $\mathcal S$ such that $P_s\nu_{s}(\varphi)=\nu_{s}(P_s \varphi),$ 
for continuous functions $\varphi$ on $\mathcal S$, and $P^*_s\nu^*_{s}$ is defined similarly.   
The following result was proved in \cite{BDGM2014, BS2016} for positive matrices, 
and in \cite{GE2016} for invertible matrices.
\begin{proposition} \label{transfer operator}
Assume that $\mu$ satisfies 
either conditions  \ref{Aexp},  \ref{A1} for invertible matrices, 
or conditions \ref{Aexp},  \ref{A2} for positive matrices. 
Let $s\in I_{\mu}$. 
Then
the spectral radii $\varrho(P_{s})$ and $\varrho(P_{s}^{*})$ are both equal to $\kappa(s)$,
and there exist a unique, up to a scaling constant,
strictly positive
H\"{o}lder continuous  
function $r_{s}$
and a unique probability measure $\nu_{s}$ on $\mathcal S$ such that 
\begin{align*}
P_s r_s=\kappa(s)r_s, \quad P_s\nu_{s}=\kappa(s)\nu_{s}.
\end{align*}
Similarly, there exist a unique strictly positive 
H\"{o}lder continuous function 
$r_{s}^{\ast}$ and 
a unique probability measure $\nu_{s}^{*}$ on $\mathcal S$ such that
$$ P_{s}^{*}r_{s}^{*}=\kappa(s)r_{s}^{*}, 
\quad 
P_{s}^{*}\nu_{s}^{*}=\kappa(s)\nu_{s}^{\ast}.$$
Moreover, 
the functions
$r_s$ and $r_s^*$ are given by 
\begin{align*}
r_{s}(x)= \int_{\mathcal{S}} |\langle x, y\rangle|^{s}\nu^*_{s}(dy),
\quad
r_{s}^*(x)= \int_{\mathcal{S}} |\langle x, y\rangle|^{s}\nu_{s}(dy),
 \quad x\in \mathcal{S}.
\end{align*}
\end{proposition}
It is easy to see that the family of kernels  
$q_{n}^{s}(x,g)=\frac{|gx|^{s}}{\kappa^{n}(s)}\frac{r_{s}(g \cdot x)}{r_{s}(x)},$
$n\geq 1$
satisfies the following cocycle property:
\begin{align} \label{cocycle01}
q_{n}^{s}(x,g_1)q_{m}^{s}(g_1\!\cdot\!x, g_2)=q_{n+m}^{s}(x,g_2g_1).
\end{align}
The equation $P_sr_s=\kappa(s)r_s$ 
implies that, for any $x\in\mathcal{S}$ and $s\in I_{\mu}$, 
the probability measures 
$\mathbb Q_{s,n}^x(dg_1,\ldots,dg_n)=q_{n}^{s}(x,g_{n}\dots g_{1})\mu(dg_1)\dots\mu(dg_n),$ $n\geq 1,$
form a projective system 
on $M(d,\mathbb{R})^{\mathbb{N}}$. 
By the Kolmogorov extension theorem,
there is a unique probability measure  $\mathbb Q_s^x$ on $M(d,\mathbb{R})^{\mathbb{N}}$, 
with marginals $\mathbb Q_{s,n}^x$; 
denote by $\mathbb{E}_{\mathbb Q_s^x}$ the corresponding expectation.

If $(g_n)_{n\in \mathbb N}$ denotes      
the coordinate process on the space of trajectories 
$M(d,\mathbb{R})^{\mathbb{N}}$, then
the sequence $(g_n)_{n \geq 1}$ 
is i.i.d.\ 
with the common law $\mu$ under $\mathbb{Q}_{0}^{x}.$ 
However, for any $s\in I_{\mu}^{\circ}$ and $x\in \mathcal{S}$, 
the sequence $(g_n)_{n \geq 1}$ is Markov-dependent under the measure $\mathbb Q_s^x$.
Let  
$$X_0^x=x, \ \ X_{n}^x= G_n \!\cdot\! x, \ \ n\geq 1.$$
By the definition of $\mathbb Q_s^x$,
for any bounded measurable function $f$ on $(\mathcal S \times \mathbb R)^{n}$, 
it holds that 
\begin{align}\label{basic equ1}
 \frac{1}{ \kappa^{n}(s) r_{s}(x) }
\mathbb{E} \Big[  r_{s}(X_{n}^{x}) & |G_nx|^{s}  f\big( X_{1}^{x}, \log |G_1 x|,\dots, X_{n}^{x}, \log |G_n x| 
                    \big) \Big]   \nonumber\\
&\quad 
=\mathbb{E}_{\mathbb{Q}_{s}^{x}} \Big[ f \big( X_{1}^{x}, \log |G_1 x|,\dots, X_{n}^{x}, \log |G_n x|  \big) \Big].
\end{align}
Under the measure $\mathbb Q_s^x$, 
the process $(X_{n}^x)_{n\in \mathbb{N}}$ is a Markov chain with the transition operator given by 
\begin{align}
Q_{s}\varphi(x)=\frac{1}{\kappa(s)r_{s}(x)}P_s(\varphi r_{s})(x)
=\frac{1}{\kappa(s)r_{s}(x)}\int_{\Gamma_{\mu}} |gx|^s \varphi(g\!\cdot\! x)r_s(g\!\cdot\!x)\mu(dg).  \nonumber
\end{align}
It has been  proved in \cite{BDGM2014} for positive matrices, 
and in \cite{GE2016} for invertible matrices, that 
$Q_{s}$ has a unique invariant probability measure $\pi_{s}$ supported on $V(\Gamma_{\mu})$
and that, 
for any $\varphi\in \mathcal{C(S)}$,
\begin{align} \label{equcontin Q s limit}
\lim_{n\to\infty}Q_{s}^{n}\varphi
=\pi_{s}(\varphi),  \quad  \mbox{where}  \ 
\pi_{s}(\varphi)=\frac{\nu_{s}(\varphi r_{s})}{\nu_{s}(r_{s})}.
\end{align}
Moreover, letting 
$\mathbb{Q}_{s}=\int\mathbb{Q}_{s}^{x}\pi_{s}(dx),$
from the results of \cite{BDGM2014, GE2016}, it follows that, 
under the assumptions of Theorem \ref{main theorem1}, for any $s\in I_{\mu}$, 
we have 
$ \lim_{n\to\infty} \frac{ \log |G_nx| }{n} =\Lambda'(s),$ 
$\mathbb{Q}_{s}$-a.s.\ and
$\mathbb{Q}_{s}^{x}$-a.s., 
where 
$\Lambda'(s)=\frac{\kappa'(s)}{\kappa(s)}$.

When $s \in (-\eta_0, 0)$ for small enough $\eta_0>0$, 
define the transfer operator $P_s$ as follows: for any $\varphi \in \mathcal{C(S)}$,
\begin{align*}
P_s \varphi (x) = \int_{\Gamma_{\mu}}  \!  |g_1 x |^{s} \varphi( g_1\!\cdot\!x ) \mu(dg_1),
\quad  x \in \mathcal{S}, 
\end{align*}
which is well-defined under condition \ref{CondiMoment}. 
The following proposition is proved in \cite{XGL19}.

\begin{proposition}
\label{transfer operator s negative}
Assume that $\mu$ satisfies 
either conditions  \ref{A1}, \ref{CondiMoment} for invertible matrices, 
or conditions \ref{A2}, \ref{CondiMoment} for positive matrices. 
Then there exists $\eta_0 < \eta$ such that for any $s \in (-\eta_0, 0)$, the spectral radius $\varrho(P_{s})$ of the operator $P_s$ 
is equal to $\kappa(s)$.
Moreover there exist a unique, up to a scaling constant,
strictly positive
H\"{o}lder continuous  
function $r_{s}$
and a unique probability measure $\nu_{s}$ on $\mathcal S$ such that 
\begin{align*}
P_s r_s=\kappa(s)r_s, \quad P_s\nu_{s}=\kappa(s)\nu_{s}.
\end{align*}
\end{proposition}

Based on Proposition \ref{transfer operator s negative}, 
in the same way as for $s>0$,  
one can define the measure $\mathbb Q_s^x$ 
for negative values $s<0$ sufficiently close to $0$, 
and one can extend the change of measure formula \eqref{basic equ1} to $s<0$.
Under the measure $\mathbb Q_s^x$, 
the process $(X_{n}^x)_{n\in \mathbb{N}}$ is a Markov chain with the transition operator $ Q_s$
and the assertion \eqref{equcontin Q s limit} holds true. We refer to \cite{XGL19} for details.

\subsection{Spectral gap of the perturbed operator} \label{sec-spgappert}
Recall that the Banach space $B_{\gamma}$ consists of all $\gamma$-H\"{o}lder continuous function on $\mathcal{S}$,
where $\gamma>0$ is a fixed small constant.  
Denote by $\mathcal{L(B_{\gamma},B_{\gamma})}$ 
the set of all bounded linear operators from $\mathcal{B}_{\gamma}$ to $\mathcal{B}_{\gamma}$
equipped with the operator norm
$\left\| \cdot \right\|_{\mathcal B_{\gamma} \to \mathcal B_{\gamma}}$.
For $s\in I_\mu^\circ$ and $z \in \mathbb{C}$ with $s+ \Re z \in I_{\mu}$, 
define a family of perturbed operators $R_{s,z}$ as follows: for any $\varphi \in \mathcal{B}_{\gamma}$, 
\begin{align} \label{operator Rsz}
R_{s,z}\varphi(x) 
 = \mathbb{E}_{\mathbb{Q}_{s}^{x}} \left[ e^{ z( \log| g_1x | - q ) }\varphi(X_{1}^x) \right], 
   \quad  x \in \mathcal{S}. 
\end{align}
It follows from the cocycle property \eqref{cocycle01} that 
\begin{align*} 
R^{n}_{s,z}\varphi(x) 
 = \mathbb{E}_{\mathbb{Q}_{s}^{x}} \left[e^{ z( \log |G_n x| - nq) } \varphi(X_{n}^x) \right],
    \quad  x \in \mathcal{S}. 
\end{align*}

The following proposition collects useful assertions that we will use in the proofs of our results. 
Denote $B_\delta(0): = \{ z \in \mathbb{C}: |z| \leq \delta \}$. 

\begin{proposition} \label{perturbation thm}
Assume that $\mu$ satisfies either conditions \ref{Aexp},  \ref{A1} for invertible matrices, 
or conditions \ref{Aexp},  \ref{A2} for positive matrices. 
Then, there exists $\delta>0$ such that for any $z \in  B_\delta(0)$,
\begin{align}
\label{perturb001}
R^{n}_{s,z}=\lambda^{n}_{s,z}\Pi_{s,z}+N^{n}_{s,z}, \  n\geq 1.
\end{align}
Moreover, for any $s \in I_{\mu}^{\circ}$, 
the following assertions hold: 
 \begin{itemize}
\item[{\rm(i)}]
$\Pi_{s,z}$ is a rank-one projection for $|z| \leq \delta$, with 
$\Pi_{s,0}(\varphi)(x)=\pi_{s}(\varphi)$ for any $\varphi \in \mathcal{B}_{\gamma}$ and $x\in \mathcal{S}$,
$\Pi_{s,z}N_{s,z}=N_{s,z}\Pi_{s,z}=0$ 
and
\begin{equation}\label{relationlamkappa001}
\lambda_{s,z} = e^{-qz} \frac{\kappa(s+z)}{\kappa(s)}, \quad \mbox{for} \  z \in  B_\delta(0).
\end{equation}
For any fixed $k \geq 1$, there exist $\varkappa_s \in(0,1)$ and $c_s$ such that
$$
\sup_{|z| < \delta}
\|\frac{d^{k}}{dz^{k}}N^{n}_{s,z} \|_{\mathcal{B}_{\gamma}\rightarrow\mathcal{B}_{\gamma}} 
\leq c_s \varkappa_s^{n}, \  n\geq 1. 
$$ 
In addition, 
the mappings $z \mapsto \Pi_{s,z}: B_\delta(0) \to \mathcal{L(B_{\gamma},B_{\gamma})}$
and $z \mapsto N_{s,z}: B_\delta(0) \to \mathcal{L(B_{\gamma},B_{\gamma})}$ are analytic
in the strong operator sense.

\item[{\rm(ii)}]
For any compact set $K\subseteq\mathbb{R}\backslash\{0\}$,  
there exists a constant $C_{K}>0$ such that 
for any $n\geq 1$ and $\varphi\in \mathcal{B}_{\gamma}$, we have
\begin{align*}
\sup_{t\in K}\sup_{x\in \mathcal{S}}|R^{n}_{s,it}\varphi(x)|\leq e^{-nC_{K}}\sup_{x\in \mathcal{S}}|\varphi(x)|.
\end{align*}

\item[{\rm(iii)}]
 The mapping $z \mapsto \lambda_{s,z}: B_\delta(0) \to \mathbb{C}$ is analytic, and  
\begin{align*}
\lambda_{s,z}=1+\frac{\sigma_s^{2}}{2}z^{2} + \frac{ \Lambda'''(s) }{6} z^3 + o(z^{3})  \quad  as \  z \to 0,
\end{align*}
where
$$
\sigma_s^{2}= \Lambda''(s)= \lim_{n\to\infty} 
\frac{1}{n} \mathbb{E}_{\mathbb{Q}_{s}}( \log |G_n x| - nq )^{2}
$$
and 
$$
\Lambda'''(s) = \lim_{n\to\infty}\frac{1}{n}\mathbb{E}_{\mathbb{Q}_{s}}( \log |G_n x| - nq )^{3}.
 $$
In addition, if the measure $\mu$ 
is non-arithmetic, 
then the asymptotic variance $\sigma_s^{2}$ is strictly positive.
\end{itemize}
\end{proposition}  

 The assertions (i), (ii), (iii) of Proposition \ref{perturbation thm}, except \eqref{relationlamkappa001},
 have been proved in \cite{BS2016} for imaginary-valued $z \in (-i\delta, i\delta)$, 
based on the perturbation theory (see \cite{HH01}).
The assertions (i), (iii) 
can be extended to the complex-valued $z \in B_\delta(0)$ without changes in the proof in \cite{BS2016}.

The identity \eqref{relationlamkappa001} is not proved in \cite{BS2016}, but
can be obtained by using the arguments from \cite{XGL19}.
By the perturbation theory, the operator $P_{s}$ and its spectral radius $\kappa(s)$ can be extended to $P_{s+z}$ 
and the eigenvalue $\kappa(s+z)$, respectively, with $z$ in the small neighborhood of $0$,
see \cite{GE2016}.
By the definitions of  $R_{s,z}$ and $P_{z}$ using the change of measure \eqref{basic equ1}, we obtain
for any $\varphi \in \mathcal{B}_{\gamma}$, $n \geq 1$,
$s \in I_{\mu}^{\circ}$ and $z \in B_\delta(0)$,
\begin{align}\label{PfRsw01}
R_{s,z}^n (\varphi)
=  e^{ -n z \Lambda'(s)} \frac{ P_{s+z}^n (\varphi r_s) }{ \kappa^n(s) r_s  }.
\end{align}
Since $r_s$ is uniformly bounded, 
using \eqref{PfRsw01} and
 the fact that $\kappa(s+z)$ is the unique eigenvalue of $P_{s+z}$,
 we deduce \eqref{relationlamkappa001}.
  



For negative values $s<0$ sufficiently close to $0$, we can define the perturbed operator $R_{s,z}$
as in \eqref{operator Rsz}. 
The following spectral gap property of $R_{s,z}$
is established in \cite{XGL19}. 

\begin{proposition} \label{perturbation thm nrgztive s}
Assume that $\mu$ satisfies 
conditions  \ref{A1}, \ref{CondiMoment} for invertible matrices, 
or conditions \ref{A2}, \ref{CondiMoment} for positive matrices. 
Then, there exist $\eta_0 < \eta$ and $\delta>0$ 
such that for any $s \in (-\eta_0, 0)$
and $z \in  B_\delta(0)$,
\begin{align*} 
R^{n}_{s,z}=\lambda^{n}_{s,z}\Pi_{s,z}+N^{n}_{s,z}, \  n\geq 1.
\end{align*}
Moreover, for any $s \in (-\eta_0, 0)$, the assertions (i), (ii), (iii) of Proposition \ref{perturbation thm} hold true.
\end{proposition}  

\section{Proof of Theorems \ref{main theorem1} and \ref{Thm-Neg-s}} \label{sec proof of main theroem1}

\subsection{Auxiliary results} \label{secAuxres001}
We need some preliminary statements.
Following Petrov \cite{Petrov75book},
under the changed measure $\mathbb Q_s^x$,
define the Cram\'{e}r series $\zeta_s$ by
\begin{align*}
\zeta_s (t) = \frac{\gamma_{s,3} }{ 6 \gamma_{s,2}^{3/2} }  
  +  \frac{ \gamma_{s,4} \gamma_{s,2} - 3 \gamma_{s,3}^2 }{ 24 \gamma_{s,2}^3 } t
  +  \frac{\gamma_{s,5} \gamma_{s,2}^2 - 10 \gamma_{s,4} \gamma_{s,3} \gamma_{s,2} + 15 \gamma_{s,3}^3 }{ 120 \gamma_{s,2}^{9/2} } t^2
  +  \ldots,
  \end{align*}
where $\gamma_{s,k} = \Lambda^{(k)} (s)$ and $\Lambda(s) = \log \kappa(s)$. 
The following lemma gives a full expansion of $\Lambda^*(q+l)$ 
in terms of power series in $l$ in a neighborhood of $0$, 
for $q=\Lambda'(s)$ and $s \in I_{\mu}^\circ \cup (\eta_0,0)$, where $\eta_0$ 
is from Proposition \ref{perturbation thm nrgztive s}.

\begin{lemma} \label{lemmaCR001}
Assume conditions of Theorem \ref{main theorem1} or Theorem \ref{Thm-Neg-s}. 
Let $q=\Lambda'(s)$. 
Then, there exists $\delta>0$ such that, for any $|l|\leq \delta,$
\begin{align*} 
\Lambda^*(q+l) = \Lambda^{*}(q) + sl +  h_s(l), 
\end{align*}
where  
$h_s$ is linked to the Cram\'{e}r series $\zeta_s$ 
by the identity
\begin{align}\label{expan hs 01}
h_s(l) = \frac{ l^2}{2 \sigma_s^2} - \frac{l^3}{\sigma_s^3} \zeta_s( \frac{l}{\sigma_s} ).
\end{align}
\end{lemma}

\begin{proof} 
Let $(\Lambda')^{-1}$ be the inverse function of $\Lambda'.$
With the notation 
$l_s= (\Lambda')^{-1}(q+l) -s $, 
we have $\Lambda'(s+l_s) = q+l$.
By the definition of $\Lambda^*$, it follows that 
$\Lambda^*(q+l) = (s+l_s)(q+l) - \Lambda(s+l_s)$.  
This, together with $\Lambda^*(q) = sq - \Lambda(s)$ and Taylor's formula, gives
\begin{align} \label{expan Lambda 01}
h_s(l):= \Lambda^*(q+l) - \Lambda^{*}(q) - sl  
= l_sl - \sum_{k=2}^{\infty} \frac{ \Lambda^{(k)}(s) }{k!} l_s^k.
\end{align}
From $\Lambda'( s + l_s ) = q+l$ and $\Lambda'(s)=q$, we deduce that $l= \Lambda'( s + l_s ) - \Lambda'(s) $, 
so that, by Taylor's formula, 
\begin{align} \label{expan Lambda 02}
l =  \sum_{k=1}^{\infty} \frac{\Lambda^{(k+1)}(s)}{k!} l_s^k.
\end{align}
The rest of the proof is similar to that in Petrov \cite{Petrov75book} (chapter VIII, section 2).
For $|l|$ small enough, the equation \eqref{expan Lambda 02} has a unique solution $l_s$ given by
\begin{align*}
l_s = \frac{l}{ \sigma_s^2 } - \frac{ \Lambda^{(3)}(s) }{ 2 \sigma_s^6 } l^2 - 
\frac{\Lambda^{(4)}(s) \sigma_s^2 - 3(\Lambda^{(3)}(s))^2}{6 \sigma_s^{10}} l^3 + \cdots. 
\end{align*}
Together with \eqref{expan Lambda 01} and \eqref{expan Lambda 02}, this implies
\begin{align*} 
h_s(l) = \sum_{k=2}^{\infty} 
\Lambda^{(k)}(s) \frac{k-1}{k!} l_s^{k}
=\frac{ l^2}{2 \sigma_s^2} - \frac{l^3}{\sigma_s^3} \zeta_s( \frac{l}{\sigma_s} ). 
\end{align*}
\end{proof}

Let us fix a non-negative Schwartz function $\rho$ on $\mathbb{R}$  with $\int_{\mathbb{R}} \rho(y) dy=1$, 
whose Fourier transform $\widehat{\rho}$ is supported on $[-1,1]$
and has a continuous extension in the complex plane. 
Moreover, $\widehat{\rho}$ is analytic in the domain 
$D : = \{ z \in \mathbb{C}: |z| < 1, \Im z \neq 0 \}$.
Such a function can be constructed as follows.
On the real line define $\widehat{\varsigma}(t)= e^{- \frac{1}{1-t^2}}$ if $t \in [-1,1]$, 
and $\widehat{\varsigma} =0$ elsewhere.
The function $\widehat{\varsigma}$ is compactly supported and has finite derivatives  of all orders. 
Its inverse Fourier transform $\varsigma$, however, is not non-negative.
Let $\widehat{\rho}_0= \widehat{\varsigma} \ast \widehat{\varsigma}$ be the convolution of 
$\widehat{\varsigma}$ with itself.
It is supported by $[-2,2]$ and its inverse Fourier transform $\rho_0$ satisfies $\rho_0 = 2\pi \varsigma^2 \geq 0$.
We show below that $\widehat{\rho}_0$ has a continuous extension in the complex plane,  
and $\widehat{\rho}_0$ is analytic in the domain $D$. 
Finally we rescale and renormalize $\rho_0$ by setting 
$\rho(y)= \rho_0(y/2)/ [ 2\widehat\rho_0(0)]$ for $y \in \mathbb{R}$.

\begin{lemma}\label{LemAnalyExten}
$\widehat{\rho}_0$ has a continuous extension in the complex plane,
and $\widehat{\rho}_0$ is analytic in the domain $D$. 
\end{lemma}
\begin{proof}
The function $\widehat{\varsigma}$ can be extended to the complex plane as follows:
\begin{equation*}
\widehat{\varsigma}(z) = 
\begin{cases}
e^{-\frac{1}{1 - z^2}}  &  |z| < 1,  \  z \in \mathbb{C} \\
0                       &  |z| \geq 1,   \   z \in \mathbb{C}.
\end{cases}
\end{equation*}
It is easily verified that $\widehat{\varsigma}$ is continuous in the interior of the unit disc and outside it, 
but is not continuous at any point on the unit circle $| z | =1$. 
Note also that $\widehat{\varsigma}$ is uniformly bounded on $\mathbb{C}$. 
Recall that the function $\widehat{\rho}_0  = \widehat{\varsigma} \ast \widehat{\varsigma}$ is defined on the real line.
We extend it to the complex plane by setting 
$\widehat{\rho}_0(z) 
= \int_{-1}^{1} \widehat{\varsigma}(t) \widehat{\varsigma}(z-t)  \mathbbm{1}_{ \{ |z-t| < 1 \} } dt.$
The latter integral  is well defined for any $z \in \mathbb{C}$,
since $\widehat{\varsigma}$ is bounded. 
We are going to show that $\widehat{\varsigma}$ is continuous in $\mathbb C$.
For any fixed $z \in \mathbb{C}$ and $h \in \mathbb{C}$ with $|h|$ small, we write 
\begin{align}\label{EquaLemContZ}
| \widehat{\rho}_0 (z + h) - \widehat{\rho}_0(z) |
\leq   \int_{-1}^{1}  \widehat{\varsigma}(t) 
  | \widehat{\varsigma}(z-t + h)  - \widehat{\varsigma}(z-t) |   dt. 
\end{align}
The set $T_{z}=\{t: |z-t| =1\}$ of points of discontinuity of the function $t\mapsto \widehat{\varsigma}(z-t)$ 
consists of at most two points. 
For any $t \in [-1, 1]$, $ t\not\in  T_{z} $, 
by the definition of $\widehat{\varsigma}$, we have that 
$| \widehat{\varsigma}(z-t + h)  - \widehat{\varsigma}(z-t) | \to 0$ as $|h| \to 0$. 
Since the Lebesgue measure of $T_{z}$ is $0$, 
applying the Lebesgue dominated convergence theorem and taking into account the boundedness of  
the function $\widehat{\varsigma}$ on $\mathbb{C}$,
we see that $\widehat{\rho}_0$ is continuous in the complex plane.

We next show that $\widehat{\rho}_0$ is analytic in the domain $D = \{ z' \in \mathbb{C}: |z'| < 1, \Im z' \neq 0 \}$. 
Fix $z\in D$. Let $\varepsilon=\Im z /2 \in (0,\frac{1}{2})$. 
Denote $D (\varepsilon) =: \{ z' \in D: |\Im z'| > \varepsilon \}$. 
One can verify that the derivative $\widehat{\varsigma}'(z)$ exists and is uniformly bounded by $\frac{c}{\varepsilon^4}$
on the domain $D (\varepsilon)$.
For any 
$h \in \mathbb{C}$ with $|h|$ small enough, we have
\begin{align*}
\frac{\widehat{\rho}_0 (z + h) - \widehat{\rho}_0(z)}{h}
=  &  \  \int_{[-1,1] \setminus T_{z}}   \widehat{\varsigma}(t) 
  \frac{\widehat{\varsigma}(z-t + h)  
         - \widehat{\varsigma}(z-t) 
          }{h}  dt  \nonumber\\
=  &  \  \int_{[-1,1] \setminus T_{z}}  \widehat{\varsigma}(t)  \left(  \int_0^1 \widehat{\varsigma}'(z-t + \theta h)
     \mathbbm{1}_{ \{ |z-t + \theta h | < 1 \} }   d\theta  \right)  dt. 
\end{align*}
Since for any $t \in [-1, 1]$ and $\theta \in [0,1]$,
we have $|\Im (z - t + \theta h)| \geq \varepsilon$ uniformly in $|h|< \varepsilon$. 
 This implies that $z - t + \theta h \in D(\varepsilon)$
 and thus $\widehat{\varsigma}'(z-t + \theta h)$ is bounded, uniformly in $|h|< \varepsilon$ and $t\in [-1,1]$. 
 Applying twice the Lebesgue dominated convergence theorem, we obtain 
 that $\widehat{\rho}_0'(z)$ exists and is given by 
 $\widehat{\rho}_0'(z) = 
 \int_{[-1,1] \setminus T_{z}}  \widehat{\varsigma}(t) \widehat{\varsigma}'( z- t) 
 dt$. 
 Hence $\widehat{\rho}_0$ is analytic in the domain $D$. 
\end{proof}

For any $\varepsilon>0$, define the density 
$\rho_{\varepsilon}(y)=\frac{1}{\varepsilon}\rho(\frac{y}{\varepsilon})$, $y\in\mathbb R,$ 
whose Fourier transform has a compact support in $[-\varepsilon^{-1},\varepsilon^{-1}]$ and is analytically extendable in a neighborhood of $0$. 
For any non-negative integrable function $\psi$,
following the paper \cite{GE2017}, we introduce two modified functions related to $\psi$ and establish some two-sided bounds.
For any $\varepsilon>0$ and $y\in \mathbb{R}$, 
set $\mathbb{B}_{\varepsilon}(y)=\{y' \in\mathbb{R}: |y'-y|\leq\varepsilon\}$ 
and  
\begin{align}\label{smoo001}
{\psi}^+_{\varepsilon}(y)=\sup_{y'\in\mathbb{B}_{\varepsilon}(y)}\psi(y') 
\quad  \text{and}  \quad 
{\psi}^-_{\varepsilon}(y)=\inf_{y'\in\mathbb{B}_{\varepsilon}(y)}\psi(y'). 
\end{align}
\begin{lemma}  \label{estimate u convo}
Suppose that  $\psi$ is a non-negative integrable function and that 
${\psi}^+_{\varepsilon}$ and ${\psi}^-_{\varepsilon}$ are measurable for any $\varepsilon>0$,
then for sufficiently small $\varepsilon$, 
there exists a positive constant $C_{\rho}(\varepsilon)$ with $C_{\rho}(\varepsilon) \to 0$ as $\varepsilon \to 0$,
such that, for any $x\in \mathbb{R}$, 
\begin{align}
{\psi}^-_{\varepsilon}\!\ast\!\rho_{\varepsilon^2}(x) - 
\int_{|y|\geq \varepsilon} {\psi}^-_{\varepsilon}(x-y) \rho_{\varepsilon^2}(y)dy
\leq \psi(x) \leq (1+ C_{\rho}(\varepsilon))
{\psi}^+_{\varepsilon}\!\ast\!\rho_{\varepsilon^2}(x). \nonumber
\end{align}
\end{lemma}
The proof of the above lemma, being similar to that of Lemma 5.2 in \cite{GLE2017}, will not be detailed here.

The next assertion is the key point in establishing Theorem \ref{main theorem1}. 
Its proof is based on 
the spectral gap properties of the perturbed operator $R_{s,z}$ 
(see Proposition \ref{perturbation thm})  
and on the saddle point method,
see Daniels \cite{Daniels1954}, Richter \cite{Richter1957}, 
Ibragimov and Linnik \cite{IbrLinnik65} and Fedoryuk \cite{Fedoryuk1987}. 
Let us introduce the necessary notation.
In the following, 
let $\varphi$ be a $\gamma$-H\"{o}lder continuous function on $\mathcal S$. 
Assume that 
$\psi: \mathbb R \mapsto \mathbb C$
is a continuous function with compact support in $\mathbb{R}$, and moreover,  
$\psi$ has a continuous extension 
in some neighborhood of $0$ in the complex plane
and can be extended analytically to the domain 
$D_{\delta} : = \{ z \in \mathbb{C}: |z| < \delta, \Im z \neq 0 \}$ for some small $\delta >0$. 
Recall that $\pi_s$ is the invariant measure of the Markov chain $X_n^x$ under the changed measure 
$\mathbb{Q}_s^x$, see \eqref{equcontin Q s limit}. 

\begin{proposition} \label{Prop Rn limit1}
Assume conditions of Theorem \ref{main theorem1}. 
Let $q=\Lambda'(s)$, where $s\in I_\mu^\circ.$ 
Then, for any 
positive sequence $(l_n)_{n \geq 1}$ satisfying $l_n \to 0$ as $n \to \infty$, 
we have, 
uniformly in $x\in \mathcal{S}$, $|l|\leq l_n $ and $\varphi \in \mathcal{B}_{\gamma}$, 
\begin{align*} 
& \Big| \sqrt{n}  \   \sigma_s   e^{n h_s(l)}
\int_{\mathbb R} e^{-it l n} R^{n}_{s,it}(\varphi)(x) \psi (t) dt
- \sqrt{2\pi} \psi(0)\pi_{s}(\varphi)
\Big|  \nonumber\\
\leq  &\  C \| \varphi \|_\gamma  \Big( \frac{ \log n }{ \sqrt{n} } + l_n \Big). 
\end{align*}
\end{proposition}

\begin{proof}
Denote $c_s(\psi)= \frac{\sqrt{2\pi}}{\sigma_s} \psi(0)\pi_{s}(\varphi)$. 
Taking sufficiently small $\delta >0$, 
we write 
\begin{align} \label{Thm1 integral1 J}
 &  
\Big| \sqrt{n}  \   e^{n h_s(l)}
\int_{\mathbb R} e^{-it l n} R^{n}_{s,it}(\varphi)(x) \psi (t) dt 
- c_s(\psi)  \Big|  \nonumber\\ 
&\leq    
\Big| \sqrt{n} ~   e^{nh_s(l)}
\int_{|t|\geq\delta}
e^{-itln}R^{n}_{s,it}(\varphi)(x) \psi(t) dt  \Big|
\nonumber\\
& \ \ + 
\Big| \sqrt{n}  \  e^{n h_s(l)}\int_{|t|<\delta}
e^{-itl n}R^{n}_{s,it}(\varphi)(x)
 \psi(t)dt - c_s(\psi)
\Big| 
\nonumber\\
& =    I(n) + J(n).  
\end{align}
For $I(n)$, 
since $\psi$ is bounded and compactly supported on the real line, 
taking into account Proposition \ref{perturbation thm} (ii), the fact $|e^{-it l n}| = 1$
and equality \eqref{expan hs 01}, 
we get
\begin{align} \label{Thm1 integral1 J1}
\sup_{x\in \mathcal{S}} \sup_{|l|\leq l_n }
|I(n)| \leq C_{\delta}  e^{- c_{\delta} n} \| \varphi \|_{\gamma}. 
\end{align}
For $J(n)$, by Proposition \ref{perturbation thm} (i), we have
$$
R^{n}_{s,it}(\varphi)(x)
= \lambda^{n}_{s,it}\Pi_{s,it}(\varphi)(x)
+ N^{n}_{s,it}(\varphi)(x).
$$
Set for brevity 
$\psi_{s,x}(t) = \Pi_{s,it}(\varphi)(x) \psi(t)$. 
It follows that 
\begin{eqnarray} \label{Prop Rn1}
J(n)
\leq \!\!\!\!\!\!\!\!&&
\Big| \sqrt{n}  \  e^{n h_s(l) }  \int_{|t|<\delta} e^{-i t l n}
\lambda^{n}_{s,it}
\psi_{s,x}(t) dt
- c_s(\psi)
 \Big|   \nonumber\\
\!\!\!\!\!\!\!\!&&
+ 
\Big| \sqrt{n}  \   e^{n h_s(l) } \int_{|t|< \delta }  e^{-i t l n}
N^{n}_{s,it}(\varphi)(x) \psi(t) dt
\Big|    \nonumber\\
=\!\!\!\!\!\!\!\!&& J_1(n)+J_2(n). 
\end{eqnarray}
For the second term $J_2(n)$, 
applying Proposition \ref{perturbation thm} (i), 
we get that there exist constants $c_{\delta}>0$ and $\varkappa \in (0,1)$ 
such that 
\begin{align*}
\sup_{x\in \mathcal{S}} \sup_{|t| < \delta}
|N^{n}_{s,it}(\varphi)(x)|
\leq 
\sup_{|t| < \delta}
\|N^{n}_{s,it}\|_{\mathcal B_{\gamma} \to \mathcal B_{\gamma}} \| \varphi \|_{\gamma}  
\leq  c_{\delta} \varkappa^n \| \varphi \|_{\gamma}.  
\end{align*}
Combining this with the continuity of the function $\psi$ at the point $0$ and the fact
$|e^{-i t l n}| = 1$, we obtain that,
uniformly in $|l| \leq l_n$, $x \in \mathcal{S}$ and $\varphi \in \mathcal{B}_{\gamma}$, 
\begin{align}  \label{SaddleIntegral 2}
J_2(n)
\leq C_{\delta}  e^{- c_{\delta} n} \| \varphi \|_{\gamma}. 
\end{align}
For the first term $J_1(n)$, 
we shall use the method of steepest descends to derive a precise asymptotic expansion.
We make a change of variable $z = i t$ to rewrite $J_1(n)$ as an integral over the 
complex interval $L_0=(-i\delta,i\delta):$
\begin{align} \label{SaddleInte J21}
J_{1}(n) 
=   
\Big| -i \sqrt{n}
\  e^{nh_s(l) } 
\int_{- i \delta}^{i \delta} e^{n (K_s(z) - zl)} \psi_{s,x}(-iz) dz - c_s(\psi)
\Big|, 
\end{align}
where $K_s(z)=\log \lambda_{s,z}$ (we choose the branch where $K_s(0)=0$), 
which is an analytic function for $|z| \leq \delta$ 
by Proposition \ref{perturbation thm} (iii).
Since the function $z \mapsto e^{n (K_s(z) - zl)}$ is analytic in the neighborhood of $0$,
and the function $z \mapsto \psi_{s,x}(-iz)$ has an analytic extension in the domain 
$D_{\delta} : = \{ z \in \mathbb{C}: |z| < \delta, \Im z \neq 0 \}$ and has a continuous extension in the domain
$\overline{D}_{\delta} : = \{ z \in \mathbb{C}: |z| \leq \delta \}$, 
by Cauchy's integral theorem 
we can choose a special path of the integration
which passes through the saddle point of the function $K_s(z) - zl$. 
From \eqref{relationlamkappa001}, 
we have
$$K_s(z) = -qz + \log \kappa(s+z) - \log \kappa(s),$$
which implies that for $|z| < \delta$,
\begin{align} \label{Expan Ks 01}
K_s(z) = \sum_{k=2}^{\infty} \gamma_{s,k} \frac{z^k}{k!},
\end{align}
where $\gamma_{s,k} = \Lambda^{(k)}(s)$ and $\Lambda(s) = \log \kappa(s)$. 
From this Taylor's expansion 
and the fact that $\Lambda^{(2)}(s)= \sigma_s^2>0$, 
it follows that the function $K_s(z) - zl$ is convex in the neighborhood of $0$. 
Consider the saddle point equation
\begin{align} \label{SaddleEqu}
K_s'(z)-l =0.
\end{align}
An equivalent formulation of \eqref{SaddleEqu} is 
$l = \sum_{k=2}^{\infty} \gamma_{s,k} \frac{z^{k-1}}{(k-1)!}$, which
by simple series inversion techniques gives the following solution:
\begin{align}  \label{SoluSaddleEqu}
z_0=z_0(l) := \frac{l}{ \gamma_{s,2} } - \frac{ \gamma_{s,3} }{ 2 \gamma_{s,2}^3 } l^2 
- \frac{ \gamma_{s,4} \gamma_{s,2} -  3 \gamma_{s,3}^2 }{6 \gamma_{s,2}^5 } l^3 + \cdots.  
\end{align}
From \eqref{SoluSaddleEqu}, 
it follows that the solution $z_0=z_0(l)$ is real for sufficiently small $l$
and that $z_0=z_0(l)\to 0$ as $l\to 0.$
Moreover, $z_0>0$ for sufficiently small $l>0$, and $z_0< 0$ for sufficiently small $l<0$. 
By Cauchy's integral theorem, $J_{1}(n)$ can be rewritten as 
\begin{align*}
J_{1}(n) =  
\Big| -i \sqrt{n}
\  e^{n h_s(l) } 
\Big\{\int_{L_1} + \int_{L_2} + \int_{L_3} \Big\} e^{n (K_s(z) - zl)} \psi_{s,x}(-iz) dz - c_s(\psi)
\Big|, 
\end{align*}
where $L_1 = (-i \delta, z_0 - i \delta)$, 
$L_2 = (z_0 - i \delta, z_0 + i \delta)$ and $L_3 = (z_0 + i \delta,  i \delta)$.
By \eqref{Expan Ks 01}, we get $K_s(it) = -\frac{1}{2} \sigma_s^2 t^2 + O(t^3)$,
which implies that $|e^{n K_s(it)}| \leq e^{-\frac{n}{3} \sigma_s^2 t^2  }$, when $t$ is sufficiently small.
Combining this with \eqref{SoluSaddleEqu} and the continuity of $K_s(z)$ in the neighborhood of $0$ 
yields that, for sufficiently small $l$, $|e^{n K_s(z)}| \leq e^{-\frac{n}{4} \sigma_s^2 \delta^2  }$,
for any $z \in L_1 \cup L_3$.
Since, for sufficiently small $l$, $l z_0>0$, we get that, 
for $z \in L_1 \cup L_3$, 
$|e^{- nzl}| = |e^{-nl z_0}| \leq 1$.  
Moreover, using the continuity of the function $z \mapsto \psi_{s,x}(-iz)$ 
in a small neighborhood of $0$ in the complex plane, 
there exists a constant $C_s>0$ such that, 
on $L_1$ and $L_3$, we have $\sup_{x \in \mathcal{S}}|\psi_{s,x}(-iz)| \leq C_s \| \varphi \|_{\gamma}$. 
Therefore, we obtain, 
for $n$ sufficiently large,  
uniformly in $|l|\leq l_n$ and $x\in \mathcal{S}$, 
\begin{align*}
\Big| -i \sqrt{n} ~
e^{n h_s(l)} 
\Big\{\int_{L_1}  + \int_{L_3} \Big\} e^{n (K_s(z) - zl)} \psi_{s,x}(-iz) dz
\Big|
\leq  O(e^{-\frac{n}{5} \sigma_s^2 \delta^2 }) \| \varphi \|_{\gamma}.
\end{align*}
It follows that 
\begin{align*}
J_{1}(n) \leq & \   
\Big| -i \sqrt{n} ~
e^{n h_s(l) } 
\int_{z_0 - i \delta}^{z_0 + i \delta} e^{n (K_s(z) - zl)} \psi_{s,x}(-iz) dz - c_s(\psi)
\Big|  \nonumber\\
& \   +  O(e^{-\frac{n}{5} \sigma_s^2 \delta^2 }) \| \varphi \|_{\gamma}. 
\end{align*}
Without loss of generality, assume that $n\geq 3$.
Making a change of variable $z= z_0 + it$ gives
\begin{align}  \label{SaddleIntegral}
&J_{1}(n) \nonumber 
\leq   
\Big| \sqrt{n} ~
e^{n h_s(l)} 
\int_{-\delta}^{\delta} e^{n [K_s(z_0 + it)-(z_0 + it)l ]} \psi_{s,x}(t-iz_0) dt - c_s(\psi)
\Big| \nonumber\\
&\qquad\qquad\qquad\qquad\qquad\qquad\qquad\qquad\qquad\qquad
   + O(e^{-\frac{n}{5} \sigma_s^2 \delta^2 }) \| \varphi \|_{\gamma}  \nonumber\\
&\leq  
\Big| \sqrt{n} \,  
e^{n h_s(l)} 
\int_{n^{-\frac{1}{2}} \log n \leq |t|< \delta}
e^{n [K_s(z_0 + it)-(z_0 + it)l ]} \psi_{s,x}(t-iz_0) dt \Big|  \nonumber\\
& \ \  \  + 
\Big| \sqrt{n} 
e^{n h_s(l)} \int_{|t|< n^{-\frac{1}{2}} \log n }  
e^{n [K_s(z_0 + it)-(z_0 + it)l ]} \psi_{s,x}(t-iz_0) dt - c_s(\psi)
\Big|     \nonumber\\ 
& \  \  \   + O(e^{-\frac{n}{5} \sigma_s^2 \delta^2 }) \| \varphi \|_{\gamma}.
\end{align}
From \eqref{SaddleEqu} and \eqref{SoluSaddleEqu}, we have $K_s'(z_0)=l$.
By Taylor's formula, we get that for $|t|< \delta$, 
\begin{align*}
K_s(z_0 + it)-(z_0 + it)l = K_s(z_0) - z_0l + \sum_{k=2}^{\infty} \frac{K_s^{(k)}(z_0) (it)^k }{k!}. 
\end{align*}
Using $K_s'(z_0)=l$ and \eqref{Expan Ks 01}, it follows that
$$K_s(z_0) - z_0 l = K_s(z_0) - z_0 K_s'(z_0) = -\sum_{k=2}^{\infty}\frac{k-1}{k!} \gamma_{s,k} z_0^k. $$
Combining this with \eqref{SoluSaddleEqu} and Lemma \ref{lemmaCR001} gives
$K_s(z_0) - z_0 l = -h_s(l)$.
Thus
\begin{align}\label{Rela Ks hs}
K_s(z_0 + it)-(z_0 + it)l = -h_s(l) + \sum_{k=2}^{\infty} \frac{K_s^{(k)}(z_0) (it)^k }{k!}. 
\end{align}
Since $K_s''(z_0) = \sigma_s^2 + O(z_0) > \frac{1}{2} \sigma_s^2 $, for small enough $z_0$, $\delta$ and $l$, 
we obtain that 
$\Re (\sum_{k=2}^{\infty} \frac{K_s^{(k)}(z_0) (it)^k }{k!}) < -\frac{1}{8} \sigma_s^2 t^2 $.
Therefore, using \eqref{Rela Ks hs} and the fact that uniformly in $x \in \mathcal{S}$, 
the function $z \mapsto \psi_{s,x}(z)$ is continuous in a neighborhood of $0$ in the complex plane, 
we obtain that, uniformly in $x\in \mathcal{S}$ and $|l|\leq l_n$, 
\begin{align*}
& \  
\Big| \sqrt{n} ~
e^{n h_s(l)} 
\int_{n^{-\frac{1}{2}} \log n \leq |t|< \delta}
e^{n [K_s(z_0 + it)-(z_0 + it)l ]} \psi_{s,x}(t-iz_0) dt 
\Big|  \\
\leq & \  c_1 \sqrt{n} 
\int_{n^{-\frac{1}{2}} \log n \leq |t|< \delta} e^{-\frac{1}{8} n \sigma_s^2 t^2} dt \| \varphi \|_{\gamma}
= O(e^{-c \log^2 n}) \| \varphi \|_{\gamma}. 
\end{align*}
This, together with \eqref{SaddleIntegral}-\eqref{Rela Ks hs}, implies
\begin{align*}
J_{1}(n) &\leq  \sup_{x\in \mathcal{S}}
\Big| \sqrt{n}
  \int_{|t|< n^{-\frac{1}{2}} \log n }  
e^{n \sum_{k=2}^{\infty} \frac{K_s^{(k)}(z_0) (it)^k }{k!} } \psi_{s,x}(t-iz_0) dt - c_s(\psi)
\Big|    \nonumber\\
&\quad   
    + O(e^{-c \log^2 n}) \| \varphi \|_{\gamma}. 
\end{align*}
Noting that $\Pi_{s,0}(\varphi)(x)=\pi_{s}(\varphi)$ and $\psi_{s,x}(0) = \psi(0) \pi_s(\varphi)$, 
we write
\begin{align}  \label{SaddleIntegral J2}
J_1(n) \leq & \   \sup_{x\in \mathcal{S}}
\Big| \sqrt{n}
 \int_{|t|< n^{-\frac{1}{2}} \log n } 
\Big( e^{n \sum_{k=2}^{\infty} \frac{K_s^{(k)}(z_0) (it)^k }{k!} } - e^{-\frac{n\sigma_s^2 t^2}{2} } \Big)
 \psi_{s,x}(t-iz_0) dt
\Big|  \nonumber\\ 
&\ +  
\sup_{x\in \mathcal{S}}  \Big| \sqrt{n} \int_{|t|< n^{-\frac{1}{2}} \log n }  e^{-\frac{n\sigma_s^2 t^2}{2}} 
\big[ \psi_{s,x}(t-iz_0) - \psi_{s,x}(0) \big]  dt
\Big|     \nonumber\\
&\ +    \sqrt{n} \psi(0) \pi_s(\varphi) \int_{|t| \geq  n^{-\frac{1}{2}} \log n } e^{-\frac{n\sigma_s^2 t^2}{2}} dt
 +   O(e^{-c \log^2 n}) \| \varphi \|_{\gamma} \nonumber\\
= & \    J_{11}(n) + J_{12}(n) + J_{13}(n) + O(e^{-c \log^2 n}) \| \varphi \|_{\gamma}.
\end{align}
We give a control of $J_{11}(n)$. 
Note that $| \psi_{s,x}(t-iz_0) |$ is bounded by $C_s \| \varphi \|_{\gamma}$, 
uniformly in $|t|< n^{-\frac{1}{2}} \log n$. 
Note also that for $|t|< n^{-\frac{1}{2}} \log n$ and for large enough $n$, we have
$| e^{ n \Re{ \sum_{k=3}^{\infty} \frac{K^{(k)}(z_0) (it)^k }{k!} } }| \leq  e^{ c n t^4} \leq C$. 
Hence using the inequality $|e^{z} - 1| \leq e^{\Re{z}} |z|$ yields
\begin{align}\label{SaddleIntegral J22}
J_{11}(n) \leq   C_s \| \varphi \|_{\gamma} \sqrt{n} \int_{|t|< n^{-\frac{1}{2}} \log n }  
     e^{-\frac{n\sigma_s^2 t^2}{2} }  n |t|^3 dt  
     \leq  \frac{C_s}{\sqrt{n}} \| \varphi \|_{\gamma}. 
\end{align}
Now we  control $J_{12}(n)$. 
Recalling that $z_0=z_0(l) \leq c_s l_n$, 
using the fact that uniformly with respect to $x \in \mathcal{S}$, 
the map $z\mapsto\psi_{s,x}(z)$ is continuous in the neighborhood of $0$ in the complex plane,
we get that for $|t| \leq n^{-\frac{1}{2}} \log n$, 
$$\sup_{x \in \mathcal{S}} | \psi_{s,x}(t-iz_0) - \psi_{s,x}(0) | 
 < c_s (n^{-\frac{1}{2}} \log n + l_n) \| \varphi \|_{\gamma}.$$ 
We then obtain
\begin{align*} 
J_{12}(n) \leq  c_s (n^{-\frac{1}{2}} \log n + l_n) \| \varphi \|_{\gamma}.
\end{align*}
It is easy to see that $J_{13}(n) \leq C \| \varphi \|_{\gamma} e^{ - c_s \log^2 n}$. 
This, together with \eqref{SaddleIntegral J2}-\eqref{SaddleIntegral J22}, 
proves that $J_1(n) \leq c_s (n^{-\frac{1}{2}} \log n + l_n) \| \varphi \|_{\gamma}.$ 
The desired result follows by combining this with \eqref{Thm1 integral1 J}-\eqref{SaddleIntegral 2}.
\end{proof}

Assume that the functions $\varphi$ and $\psi$ satisfy the same properties as in Proposition \ref{Prop Rn limit1}.
The following result, for $s<0$ small enough, will be used to prove Theorem \ref{Thm-Neg-s}. 

\begin{proposition} \label{Prop Rn limit2}
Assume conditions of Theorem \ref{Thm-Neg-s}. 
Then, there exists $\eta_0 < \eta$ such that for any $s \in (-\eta_0, 0)$, $q=\Lambda'(s)$ and  for any 
positive sequence $(l_n)_{n \geq 1}$ satisfying $l_n \to 0$ as $n \to \infty$, 
we have, uniformly in $x\in \mathcal{S}$, $|l|\leq l_n $ and $\varphi \in \mathcal{B}_{\gamma}$, 
\begin{align*} 
& \Big| \sqrt{n}  \   \sigma_s   e^{n h_s(l)}
\int_{\mathbb R} e^{-it l n} R^{n}_{s,it}(\varphi)(x) \psi (t) dt
- \sqrt{2\pi} \psi(0)\pi_{s}(\varphi)
\Big|  \nonumber\\
\leq  &\  C \| \varphi \|_\gamma  \Big( \frac{ \log n }{ \sqrt{n} } + l_n \Big). 
\end{align*}
\end{proposition}

\begin{proof}
Using Propositions \ref{transfer operator s negative} and \ref{perturbation thm nrgztive s},
the proof of Proposition \ref{Prop Rn limit2} can be carried out as 
the proof of Proposition \ref{Prop Rn limit1}. We omit the details. 
\end{proof}


\subsection{Proof of  Theorem \ref{main theorem1} } \label{proof T21}
Recall that
$q=\Lambda'(s)$, $\Lambda^*(q+l)=\Lambda^{*}(q) + sl + h_s(l)$, 
$x\in \mathcal{S},$  and $|l|\leq l_n \to 0$, 
as $n \to \infty$.
Taking into account that $e^{n\Lambda^{*}(q)}=e^{sqn}/\kappa^{n}(s)$
and using the change of measure \eqref{basic equ1}, 
we write
\begin{align} \label{Prop change measure1}
&A_n(x,l) := \sqrt{2\pi n}~s\sigma_s  
e^{n \Lambda^*(q+l)}
\frac{1}{r_s(x)}\mathbb{P}(\log |G_n x|\geq n(q+l)) \nonumber\\
= &
\sqrt{2\pi n} \, s\sigma_s e^{nsl}e^{n h_s(l)} e^{sqn}\mathbb{E}_{\mathbb{Q}_{s}^{x}}
\Big( \frac{1}{r_{s}(X^x_{n})}e^{-s \log |G_n x| } \mathbbm{1}_{\{ \log |G_nx| \geq n(q+l) \} } \Big). 
\end{align}
Setting $T_{n}^x= \log |G_n x|  
-nq$ 
and $\psi_s(y)=e^{-sy}\mathbbm{1}_{\{y\geq 0\}}$, 
from  \eqref{Prop change measure1} we get
\begin{eqnarray} \label{change measure equ1}
A_n(x,l)=\sqrt{2\pi n}~s\sigma_s e^{n h_s(l)}
\mathbb{E}_{\mathbb{Q}_{s}^{x}} \left(\frac{1}{r_{s}(X_{n}^x)}\psi_s(T_{n}^x-nl)\right).
\end{eqnarray}
\textit{Upper bound.} 
Let $\varepsilon\in(0,1)$ 
and
${\psi}^+_{s,\varepsilon}(y) = \sup_{y'\in\mathbb{B}_{\varepsilon}(y)} \psi_s(y')$
be  defined as in \eqref{smoo001} but with $\psi_s$ instead of  $\psi$. 
Using Lemma \ref{estimate u convo} leads to
\begin{align}\label{Thm1 upper1}
A_n(x,l)
&\leq (1+ C_{\rho}(\varepsilon))
\sqrt{2\pi n}~s\sigma_s 
e^{n h_s(l) }
\mathbb{E}_{\mathbb{Q}_{s}^{x}}
\left[\frac{1}{r_{s}(X_{n}^x)}
({\psi}^+_{s,\varepsilon}\!\ast\!\rho_{\varepsilon^2})
(T_{n}^x-nl)\right]
\nonumber \\
&=: B_n^+(x,l).
\end{align} 
Denote by 
$\widehat{{\psi}}^+_{s,\varepsilon}$ the Fourier transform of ${\psi}^+_{s,\varepsilon}$.
Elementary calculations give
\begin{align}  \label{Thm1 estamite1 u01}
\sup_{t\in \mathbb R} |\widehat \psi^+_{s,\varepsilon}(t)|
\leq
\widehat \psi^+_{s,\varepsilon}(0) 
= 
\int_{-\varepsilon}^{\varepsilon} dy
+
\int_{\varepsilon}^{+\infty} e^{-s(y-\varepsilon)} dy 
=\frac{1+2s\varepsilon}{s}.
\end{align}
By the inversion formula, for any $y\in \mathbb R,$ 
$$
{\psi}^+_{s,\varepsilon}\!\ast\!\rho_{\varepsilon^{2}}(y)
=\frac{1}{2\pi}\int_{\mathbb{R}}e^{ity}
\widehat {\psi}^+_{s,\varepsilon}(t) \widehat\rho_{\varepsilon^{2}}(t)dt.
$$
Substituting $y=T_{n}^x-nl$, taking expectation with respect to $\mathbb{E}_{\mathbb{Q}_{s}^{x}}$,
and using Fubini's theorem, we get
\begin{equation} \label{Fubini1}
\mathbb{E}_{\mathbb{Q}_{s}^{x}}
\Big[ \frac{1}{ r_{s}(X_{n}^x) }
( {\psi}^+_{s,\varepsilon} \!\ast\! \rho_{\varepsilon^{2}} ) (T_{n}^x-nl) \Big]
= \frac{1}{2\pi}
\int_{\mathbb{R}} e^{-itln} R^{n}_{s,it}(r_{s}^{-1})(x)
\widehat {\psi}^+_{s,\varepsilon}(t) \widehat\rho_{\varepsilon^{2}}(t) dt,
\end{equation}
where
\begin{equation}
R^{n}_{s,it}(r_{s}^{-1})(x)
=\mathbb{E}_{\mathbb{Q}_{s}^{x}}\left[e^{it T_{n}^x}\frac{1}{r_{s}(X_{n}^x)}\right].  \nonumber
\end{equation}
Note that $\widehat {\psi}^+_{s,\varepsilon} \widehat\rho_{\varepsilon^{2}}$ 
is compactly supported in $\mathbb{R}$ 
since $\widehat\rho_{\varepsilon^{2}}$ has a compact support. 
One can verify that $\widehat {\psi}^+_{s,\varepsilon}$ has an analytic extension 
in a neighborhood of $0$. By Lemma \ref{LemAnalyExten}, we see that 
the function $\widehat\rho_{\varepsilon^{2}}$ has a continuous extension in the complex plane,
and has an analytic in the domain 
$D_{\varepsilon^2} : = \{ z \in \mathbb{C}: |z| < \varepsilon^2, \Im z \neq 0 \}$. 
Using Proposition \ref{Prop Rn limit1}
with $\varphi=r_s^{-1}$ and $\psi=\widehat {\psi}^+_{s,\varepsilon} \widehat\rho_{\varepsilon^{2}}$,  
it follows that 
\begin{align} \label{Thm1 identity upper 1}
\lim_{n\to\infty}\sup_{x\in \mathcal{S}}\sup_{|l|\leq l_n }
\left| 
B_n^+(x,l)
-(1+ C_{\rho}(\varepsilon)) \pi_{s}(r_{s}^{-1}) s \widehat{\psi}^+_{s,\varepsilon}(0)
\widehat{\rho}_{\varepsilon^2}(0) 
\right| =0.
\end{align}
Since $\widehat{\rho}_{\varepsilon^2}(0)=1$,
from \eqref{change measure equ1}-\eqref{Thm1 identity upper 1}, 
we have that 
for sufficiently small $\varepsilon\in (0,1)$, 
\begin{align*} 
\limsup_{n\to\infty}\sup_{x\in \mathcal{S}}\sup_{|l|\leq l_n }
A_n(x,l) 
&\leq  (1+ C_{\rho}(\varepsilon))
s\pi_{s}(r_{s}^{-1})\widehat{\psi}^+_{s,\varepsilon}(0)\widehat{\rho}_{\varepsilon}(0) 
\\
&\leq (1+ C_{\rho}(\varepsilon)) (1+ 2s\varepsilon)\pi_{s}(r_{s}^{-1}).
\end{align*}
Letting $\varepsilon\to 0$ and noting that $C_{\rho}(\varepsilon) \to 0$, 
we obtain the upper bound: 
\begin{align} \label{Thm1 upper bound}
\limsup_{n\to\infty}\sup_{x\in \mathcal{S}}\sup_{|l|\leq l_n }
A_n(x,l)
\leq \pi_{s}(r_{s}^{-1}) = \frac{1}{ \nu_s(r_s) }.
\end{align}
\textit{Lower bound.} 
For $\varepsilon\in(0,1)$,
let
${\psi}^-_{s,\varepsilon}(y) = \inf_{y'\in\mathbb{B}_{\varepsilon}(y)} \psi_s(y')$
be defined as in \eqref{smoo001} with $\psi_s$ instead of  $\psi$.
From \eqref{change measure equ1} and Lemma \ref{estimate u convo}, we get
\begin{align} \label{Lowerbound An 1}
A_n(x,l) \geq & \
\sqrt{2\pi n}~s\sigma_s e^{n h_s(l) }
\mathbb{E}_{\mathbb{Q}_{s}^{x}}
\left[\frac{1}{r_{s}(X_{n}^x)}({\psi}^-_{s,\varepsilon}\!\ast\!\rho_{\varepsilon^2})(T_{n}^x-nl)\right]   \nonumber\\
& \ -  
\sqrt{2\pi n}~s\sigma_s e^{n h_s(l) }
\int_{|y|\geq \varepsilon}
\mathbb{E}_{\mathbb{Q}_{s}^{x}}
\left[\frac{1}{r_{s}(X_{n}^x)} {\psi}^-_{s,\varepsilon} (T_{n}^x-nl-y)\right]
\rho_{\varepsilon^2}(y)dy   \nonumber\\
: =  & \ B_n^-(x,l)- C_n^-(x,l). 
\end{align}
For the first term $B_n^-(x,l)$, 
applying \eqref{Fubini1} with ${\psi}^+_{s,\varepsilon} \rho_{\varepsilon^{2}}$ 
replaced by ${\psi}^-_{s,\varepsilon} \rho_{\varepsilon^{2}}$, we get
\begin{align*}  
B_n^-(x,l) = 
\sqrt{\frac{n}{2\pi} }~s\sigma_s e^{n h_s(l) }
\int_{\mathbb{R}} e^{-itln} R^{n}_{s,it}(r_{s}^{-1})(x)
\widehat {\psi}^-_{s,\varepsilon}(t) \widehat\rho_{\varepsilon^{2}}(t) dt.  
\end{align*}
In the same way as for the upper bound, using 
$
\widehat \psi^-_{s,\varepsilon}(0) 
= \frac{e^{-2s\varepsilon}}{s}
$
and Proposition \ref{Prop Rn limit1}
with $\varphi=r_s^{-1}$ and $\psi=\widehat {\psi}^-_{s,\varepsilon} \widehat\rho_{\varepsilon^{2}}$
(one can check that the functions $\varphi$ and $\psi$ 
satisfy the required conditions in Proposition \ref{Prop Rn limit1}), 
we obtain the lower bound: 
\begin{align} \label{Thm1 lower bound}
\liminf_{n\to\infty}\sup_{x\in \mathcal{S}}\sup_{|l|\leq l_n }
B_n^-(x,l)
\geq \pi_{s}(r_{s}^{-1}) = \frac{1}{ \nu_s(r_s) }.
\end{align}
For the second term $C_n^-(x,l)$, noting that $\psi^-_{s,\varepsilon} \leq \psi_s$ 
and applying Lemma \ref{estimate u convo} to $\psi_s$, we get
$\psi^-_{s,\varepsilon} \leq \psi_s 
\leq (1+ C_{\rho}(\varepsilon)){\psi}_{s,\varepsilon}^+ \! \ast \! \rho_{\varepsilon^2}$. 
We use the same argument as in \eqref{Fubini1} to obtain
\begin{align*}
C_n^-(x,l)
& \leq   (1+ C_{\rho}(\varepsilon))
\sqrt{2\pi n}~s\sigma_s e^{n h_s(l) }\\
&\ \quad \times \int_{|y|\geq \varepsilon}
\mathbb{E}_{\mathbb{Q}_{s}^{x}}
\left[\frac{1}{r_{s}(X_{n}^x)}
({\psi}_{s,\varepsilon}^+  \ast  \rho_{\varepsilon^2}) 
(T_{n}^x-nl-y)\right]
\rho_{\varepsilon^2}(y)
dy   \nonumber\\
& =  (1+ C_{\rho}(\varepsilon))
\sqrt{\frac{n}{2 \pi} }~s\sigma_s e^{n h_s(l) }
\\ 
&\ \quad  \times
\int_{|y|\geq \varepsilon}
\left(  \int_{\mathbb{R}} e^{-it(ln+y)} R^{n}_{s,it}(r_{s}^{-1})(x)
\widehat {\psi}^+_{s,\varepsilon}(t) 
\widehat\rho_{\varepsilon^{2}}(t)
dt \right)\rho_{\varepsilon^2}(y)
dy.
\end{align*}
Notice that, from Lemma \ref{lemmaCR001}, 
for any fixed $y \in \mathbb{R}$, 
it holds, uniformly in $l$ satisfying $|l| \leq l_n$,  that $e^{nh_s(l)-nh_s(l+\frac{y}{n})} \to 1$ as $n \to \infty$.
Applying Proposition \ref{Prop Rn limit1} again
with $\varphi=r_s^{-1}$, $\psi=\widehat \psi^+_{s,\varepsilon}\widehat\rho_{\varepsilon^{2}}$,
and using the Lebesgue dominated convergence theorem, 
we obtain
\begin{align*}
&\limsup_{n\to\infty}\sup_{x\in \mathcal{S}}\sup_{|l|\leq l_n }
C_n^-(x,l) \leq  (1+ C_{\rho}(\varepsilon)) 
s \pi_{s}(r_{s}^{-1}) \widehat \psi^+_{s,\varepsilon}(0) \widehat\rho_{\varepsilon^{2}}(0)
\int_{|y| \geq \varepsilon} \rho_{\varepsilon^2}(y)dy  \nonumber\\
&\qquad\qquad =   (1+ C_{\rho}(\varepsilon)) \pi_{s}(r_{s}^{-1}) (1+2s \varepsilon) \int_{|y|\geq \frac{1}{\varepsilon}} \rho(y)dy
\to  0, \quad  \mbox{as}  \  \varepsilon \to 0,
\end{align*}
since $\rho$ is integrable on $\mathbb{R}$.
This, together with \eqref{Lowerbound An 1}-\eqref{Thm1 lower bound}, implies the lower bound: 
\begin{align} \label{lower bound An 001}
\liminf_{n\to\infty}\sup_{x\in \mathcal{S}}\sup_{|l|\leq l_n }
A_n(x,l)
\geq \pi_{s}(r_{s}^{-1}) = \frac{1}{ \nu_s(r_s) }, 
\end{align}
as required.
We conclude the proof of Theorem \ref{main theorem1} 
by combining \eqref{Thm1 upper bound} and \eqref{lower bound An 001}.

\subsection{Proof of  Theorem \ref{Thm-Neg-s}}
Since the change of measure formula can be extended for small $s<0$,
under the conditions of Theorem \ref{Thm-Neg-s}, 
we have, similar to \eqref{Prop change measure1},
\begin{align*}
&  -s \sigma_s \sqrt{2\pi n} \,   
e^{n \Lambda^*(q+l)}
\frac{1}{r_s(x)}\mathbb{P}(\log |G_n x|\leq n(q+l)) \nonumber\\
= &  -s \sigma_s
\sqrt{2\pi n} \,  e^{nsl}e^{n h_s(l)} e^{sqn}\mathbb{E}_{\mathbb{Q}_{s}^{x}}
\Big( \frac{1}{r_{s}(X^x_{n})}e^{-s \log |G_n x| } \mathbbm{1}_{\{ \log |G_nx| \leq n(q+l) \} } \Big). 
\end{align*}
Applying Proposition \ref{Prop Rn limit2}, 
we can follow the proof of Theorem \ref{main theorem1}
to show Theorem \ref{Thm-Neg-s}. We omit the details.


\section{Proof of Theorems \ref{main theorem3} and \ref{Thm-Neg-s-Target}}\label{sec proof of main theroem3}
We first establish the following assertion which will be used to prove Theorem \ref{main theorem3},
but which is of independent interest.
Let $\psi$ be a measurable function on $\mathbb{R}$ and $\varepsilon>0.$
Denote, for brevity, 
$\psi_s(y)=e^{-sy}\psi(y)$ and
$$\psi^+_{s,\varepsilon} (y)= \sup_{y'\in\mathbb{B}_{\varepsilon}(y)} \psi_s(y'),
\quad
\psi^-_{s,\varepsilon}(y)= \inf_{y'\in\mathbb{B}_{\varepsilon}(y)} \psi_s(y').$$
Introduce the following condition:
for any $s\in I_\mu^\circ$ and $\varepsilon>0,$ the functions
 $y\mapsto \psi^+_{s,\varepsilon}(y)$ 
 and  $y\mapsto \psi^-_{s,\varepsilon}(y)$ 
 are measurable and
\begin{align}\label{condition g}
\lim_{\varepsilon\to0^{+}}\int_{\mathbb{R}} \psi^+_{s,\varepsilon}(y)
dy
=\lim_{\varepsilon\to0^{+}}\int_{\mathbb{R}} \psi^-_{s,\varepsilon}(y)
dy 
=\int_{\mathbb{R}}e^{-sy}\psi(y)dy<+\infty.
\end{align}
\begin{theorem}  \label{main theorem2}
Suppose the assumptions of Theorem \ref{main theorem1} hold true.
Let $q=\Lambda'(s)$, where $s\in I_\mu^\circ$.
Assume that $\varphi$ is a H\"{o}lder continuous function on $\mathcal{S}$
and $\psi$ is a measurable function on $\mathbb{R}$ 
satisfying condition \eqref{condition g}. 
Then, for any positive sequence $(l_n)_{n \geq 1}$ satisfying $\lim_{n\to \infty}l_n = 0$, 
we have 
\begin{align}\label{Asy-s-Posi}
&\lim_{n\to\infty}  \sup_{x\in \mathcal{S}} \sup_{|l|\leq l_n }
\Bigg|
\sqrt{2\pi n} \, \sigma_s e^{n \Lambda^*(q+l)}  
\mathbb{E}  \Big[ \varphi(X_{n}^{x}) \psi( \log |G_n x|-n(q+l) ) \Big] \nonumber\\
&\quad\quad\quad\quad\quad\quad\quad\quad\quad\quad\quad\quad
- \bar r_{s}(x) \nu_s(\varphi) \int_{\mathbb{R}} e^{-sy} \psi(y) dy 
\Bigg| 
=0. 
\end{align}
\end{theorem}
%
Before proceeding with the proof of this theorem, let us give some
examples of functions satisfying condition \eqref{condition g}.
It is easy to see that 
\eqref{condition g} holds for increasing non-negative functions $\psi$ satisfying
$\int_{\mathbb{R}}e^{-sy}\psi(y)dy<+\infty,$
in particular, 
for the indicator function
$\psi(y)=\mathbbm{1}_{\{y \geq c\}}$, $y \in \mathbb R$, where $c\in \mathbb R$ is a fixed constant.
Another example for which \eqref{condition g} holds true is 
when $\psi$ is non-negative,
continuous and there exists $\varepsilon>0$ such that
\begin{align} \label{conditiong002}
\int_{\mathbb{R}}e^{-sy}
\psi_\varepsilon^+(y)
dy<+\infty, 
\end{align}
where the function $\psi_\varepsilon^+(y)  =  \sup_{y'\in\mathbb{B}_{\varepsilon}(y)} \psi(y')$ is assumed to be measurable.

\begin{proof}[\textit{Proof of Theorem \ref{main theorem2}}]
Without loss of generality, we assume that both
$\varphi$ and $\psi$ are non-negative (otherwise, we decompose the functions $\varphi=\varphi^+ - \varphi^-$
and $\psi = \psi^+ - \psi^-$).
Let $T_{n}^x=\log|G_n x|-nq$. 
Since $e^{n\Lambda^{*}(q)}=e^{sqn}/\kappa^{n}(s)$,
using the change of measure \eqref{basic equ1}, we have
\begin{align}
A_n(x,l) := \ &\sqrt{2\pi n}~\sigma_s e^{n \Lambda^*(q+l)} 
\frac{1}{r_{s}(x)}
\mathbb{E}\Big[\varphi(X_{n}^{x})\psi(\log |G_nx|-n(q + l))\Big] \nonumber\\
=  \ &
\sqrt{2\pi n}~\sigma_s e^{n s l}e^{n h_s(l)}
e^{sqn}\mathbb{E}_{\mathbb{Q}_{s}^{x}}
\left[ (\varphi r_{s}^{-1})(X_{n}^x) e^{ -s \log |G_nx| } \psi(T_{n}^x- n l) \right] \nonumber\\
=  \ &
\sqrt{2\pi n}~\sigma_se^{n h_s(l) }
\mathbb{E}_{\mathbb{Q}_{s}^{x}}
\left[(\varphi r_{s}^{-1})(X_{n}^x)e^{-s( T_{n}^x- n l )}\psi(T_{n}^x- n l)\right].  \nonumber
\end{align}
For brevity, set $\Phi_s(x)=\left(\varphi r_{s}^{-1}\right)(x),$ $x\in \mathcal{S}$, and
$\Psi_s(y)=e^{-sy}\psi(y),$ $y\in \mathbb R$. 
Then,
\begin{align} \label{Target An}
A_n(x,l) = 
\sqrt{2\pi n}~\sigma_s e^{n h_s(l)}
\mathbb{E}_{\mathbb{Q}_{s}^{x}}\left[\Phi_s(X_{n}^x) \Psi_s(T_{n}^x-nl)\right]. 
\end{align}
\textit{Upper bound.} 
We wish to write the expectation in \eqref{Target An} as 
an integral of the Fourier transform of $\Psi_s,$
which, however, may not belong to the space $L^{1}(\mathbb{R})$.
As in the proof of Theorem \ref{main theorem1} (see Section \ref{proof T21}),
we make use of the convolution technique to overcome this difficulty.
Applying Lemma \ref{estimate u convo} to $\Psi_s$, 
one has, for sufficiently small $\varepsilon > 0$,
\begin{align}  \label{Thm2 upper01}
A_n(x,l)
\leq & \ (1+ C_{\rho}(\varepsilon))
\sqrt{2\pi n}~\sigma_s e^{n h_s(l)}
\mathbb{E}_{\mathbb{Q}_{s}^{x}}\left[\Phi_s(X_{n}^x)
({\Psi}^+_{s,\varepsilon}\!\ast\!\rho_{\varepsilon^2})(T_{n}^x-nl)\right] \nonumber\\
 : =  &  \  B_n(x,l), 
\end{align}
where ${\Psi}^+_{s,\varepsilon}(y) = \sup_{y'\in\mathbb{B}_{\varepsilon}(y)} \Psi_s(y')$, $y \in \mathbb{R}$. 
Using the same arguments as for deducing \eqref{Fubini1}, we have
\begin{align}   \label{target upper Bn}
B_n(x,l)
= (1+ C_{\rho}(\varepsilon)) \frac{\sigma_s}{\sqrt{2\pi} } \sqrt{n} \ e^{n h_s(l)} \int_{\mathbb{R}}
e^{-itln}
R^{n}_{s,it} \Phi_s(x)\widehat{\Psi}^+_{s,\varepsilon}(t) \widehat{\rho}_{\varepsilon^2}(t)dt,
\end{align}
where
$R^{n}_{s,it}\Phi_s(x)=\mathbb{E}_{\mathbb{Q}_{s}^{x}}\left[e^{it T_{n}^x}\Phi_s(X_{n}^x)\right]$
and $\widehat{\Psi}^+_{s,\varepsilon}$ is the Fourier transform of $\Psi^+_{s,\varepsilon}$. 
Note that $\Phi_s$ is strictly positive and $\gamma$-H\"{o}lder continuous function on $\mathcal S$, 
and $\widehat{\Psi}^+_{s,\varepsilon} \widehat{\rho}_{\varepsilon^2}$ 
has a compact support in $\mathbb{R}$. 
Applying Proposition \ref{Prop Rn limit1}
with $\varphi = \Phi_s$ and $\psi=\widehat{\Psi}^+_{s,\varepsilon} \widehat{\rho}_{\varepsilon^2}$
(one can verify that the functions $\varphi$ and $\psi$ 
satisfy the required conditions in Proposition \ref{Prop Rn limit1}), we obtain 
\begin{align*}
\lim_{n\to\infty}  \sup_{x\in \mathcal{S}}  \sup_{|l|\leq l_n } B_n(x,l)
= (1+ C_{\rho}(\varepsilon)) \pi_{s}(\Phi_s) \widehat{\Psi}^+_{s,\varepsilon}(0) \widehat{\rho}_{\varepsilon^2}(0).
\end{align*}
Since 
$\widehat{\Psi}^+_{s,\varepsilon}(0)
=\int_{\mathbb{R}}\sup_{y'\in\mathbb{B}_{\varepsilon}(y)}e^{-sy'}\psi(y')dy$ and 
$\widehat{\rho}_{\varepsilon^2}(0) =1$, 
letting $\varepsilon$ go to $0$,
using the condition \eqref{condition g} 
and the fact that $C_{\rho}(\varepsilon) \to 0$ as $\varepsilon \to 0$, 
we get the upper bound: 
\begin{align} \label{Thm2 upper result}
\limsup_{n\to\infty}  \sup_{x\in \mathcal{S}} \sup_{|l|\leq l_n}
A_n(x,l)
\leq \pi_{s}(\Phi_s) \int_{\mathbb{R}} e^{-sy} \psi(y) dy.
\end{align}
\textit{Lower bound.} 
Denote ${\Psi}^-_{s,\varepsilon}(y) = \inf_{y'\in\mathbb{B}_{\varepsilon}(y)} \Psi_s(y')$. 
From \eqref{Target An}, using Lemma \ref{estimate u convo}, we get
\begin{align}  \label{Target Lowerbound An 1}
A_n(x,l) 
\geq  & \ 
\sqrt{2\pi n} \  \sigma_s e^{n h_s(l) }
\mathbb{E}_{\mathbb{Q}_{s}^{x}}
\left[ \Phi_s(X_{n}^x)
({\Psi}^-_{s,\varepsilon}\! \ast \! \rho_{\varepsilon^2})
(T_{n}^x-nl)\right]   \nonumber\\
& -  
\sqrt{2\pi n}  \   \sigma_s e^{n h_s(l) }
\int_{|y|\geq \varepsilon}
\mathbb{E}_{\mathbb{Q}_{s}^{x}}
\left[ \Phi_s(X_{n}^x)
{\Psi}^-_{s,\varepsilon}
(T_{n}^x-nl-y)\right]\rho_{\varepsilon^2}(y)dy   \nonumber\\
:= &  \ 
B_n^-(x,l)- C_n^-(x,l). 
\end{align}
For $B_n^-(x,l)$, 
we proceed as for \eqref{Thm2 upper01} and \eqref{target upper Bn}, with ${\Psi}^+_{s,\varepsilon}$
replaced by ${\Psi}^-_{s,\varepsilon}$. 
Using 
Proposition \ref{Prop Rn limit1},
with $\varphi=\Phi_s$ and $\psi=\widehat \Psi^-_{s,\varepsilon} \widehat \rho_{\varepsilon^2},$ 
and the fact that $\widehat \rho_{\varepsilon^2}(0)=1$ and $\widehat{\Psi}^-_{s,\varepsilon}(0)
=\int_{\mathbb{R}}\inf_{y'\in\mathbb{B}_{\varepsilon}(y)}e^{-sy'}\psi(y')dy$, 
in an analogous way as in \eqref{Thm2 upper result}, 
we obtain that 
\begin{align}  \label{Target lower bound Bn}
&  \  \lim_{n\to\infty}\sup_{x\in \mathcal{S}}\sup_{|l|\leq l_n }  B_n^-(x,l)
\nonumber \\
=  & \  \pi_{s}(r_{s}^{-1}) \int_{\mathbb{R}}\inf_{y\in\mathbb{B}_{\varepsilon}(z)}e^{-sy}\psi(y)dz
\to  \pi_{s}(r_{s}^{-1}) \int_{\mathbb{R}}e^{-sy}\psi(y)dy, \  \mbox{as} \ \varepsilon \to 0,
\end{align}
where the last convergence is due to the condition \eqref{condition g}.  
For $C_n^-(x,l)$, 
noting that $\Psi^-_{s,\varepsilon} \leq \Psi_s$,
applying Lemma \ref{estimate u convo} to $\Psi_s$ we get
$\Psi^-_{s,\varepsilon} 
\leq (1+ C_{\rho}(\varepsilon))\widehat{\Psi}^+_{s,\varepsilon} \widehat{\rho}_{\varepsilon^2}$. 
Similarly to \eqref{target upper Bn}, we show that
\begin{align*}
C_n^-(x,l) 
&\leq  (1+ C_{\rho}(\varepsilon))
\sqrt{\frac{n}{2 \pi} }~ \sigma_s e^{n h_s(l) }
\\
&  \times\int_{|y|\geq \varepsilon}
\left(  \int_{\mathbb{R}} e^{-it(ln+y)} R^{n}_{s,it}(\Phi_s)(x)
\widehat {\Psi}^+_{s,\varepsilon}(t) \widehat\rho_{\varepsilon^{2}}(t) dt \right)
\rho_{\varepsilon^2}(y)dy.
\end{align*}
From Lemma \ref{lemmaCR001}, 
for any fixed $y \in \mathbb{R}$, 
it holds that $e^{nh_s(l)-nh_s(l+\frac{y}{n})} \to 1$, uniformly in $|l| \leq l_n$ as $n \to \infty$.
Applying Proposition \ref{Prop Rn limit1} 
with $\varphi=\Phi_s$ and $\psi=\widehat \Psi^+_{s,\varepsilon}\widehat\rho_{\varepsilon^{2}}$,
it follows, from the Lebesgue dominated convergence theorem, that 
\begin{align*}
&\limsup_{n\to\infty}\sup_{x\in \mathcal{S}}\sup_{|l|\leq l_n } C_n^-(x,l) \\
&\quad \leq    (1+ C_{\rho}(\varepsilon)) 
\pi_{s}(\Phi_s) \widehat \Psi^+_{s,\varepsilon}(0) \widehat\rho_{\varepsilon^{2}}(0)
\int_{|y| \geq \varepsilon} \rho_{\varepsilon^2}(y)dy  
\to  0
\end{align*}
as $\varepsilon \to 0$.
Combining this with \eqref{Target Lowerbound An 1}-\eqref{Target lower bound Bn}, we get the lower bound 
\begin{align} \label{Target lower bound An 02}
\liminf_{n\to\infty}\sup_{x\in \mathcal{S}}\sup_{|l|\leq l_n }
A_n(x,l)
\geq \pi_{s}(\Phi_s) \int_{\mathbb{R}}e^{-sy}\psi(y)dy. 
\end{align}
Putting together \eqref{Thm2 upper result} and \eqref{Target lower bound An 02},
and noting that $\pi_{s}(\Phi_s) = \pi_{s}(\varphi r_s^{-1}) = \frac{\nu_s(\varphi)}{\nu_s(r_s)}$,
the result follows. 
\end{proof}

In the sequel, we deduce Theorem \ref{main theorem3} from Theorem \ref{main theorem2} 
using approximation techniques.

\begin{proof}[\textit{Proof of Theorem \ref{main theorem3}}]
Without loss of generality, we assume that $\varphi\geq 0$ and $\psi \geq 0$.
Let $\Psi_s(y)=e^{-sy}\psi(y)$, $y \in \mathbb{R}$. 
We construct two step functions as follows:
for any $\eta \in(0,1)$, $m\in\mathbb{Z}$ and $y\in [m\eta, (m+1)\eta)$, set
\begin{align*}
\Psi^+_{s,\eta} (y)=\sup_{y\in [m\eta, (m+1)\eta)} \Psi_s(y) \quad \mbox{and} \quad
\Psi^-_{s,\eta} (y)=\inf_{y\in [m\eta, (m+1)\eta)} \Psi_s(y).
\end{align*}
By the definition of the direct Riemann integrability, the following two limits exist and are equal:
\begin{align}  \label{DiretRiemStep 01}
\lim_{\eta\to 0^+}\int_{\mathbb{R}} \Psi^+_{s,\eta}(y)dy
=\lim_{\eta\to 0^+}\int_{\mathbb{R}} \Psi^-_{s,\eta}(y)dy.
\end{align}
Since $\Psi_s$ is directly Riemann integrable, 
we have $M:= \sup_{y \in \mathbb{R}} \Psi_s(y)<+\infty$. 
Let $\varepsilon \in (0, M\eta)$ be fixed. 
Denote $I_{m}=[(m-1) \eta, m\eta)$, 
$I_{m}^-=\big(m\eta-\frac{\varepsilon}{M 4^{|m|}}, m\eta \big)$,
and $I_{m}^+ = \big[m\eta,  m\eta + \frac{\varepsilon}{M 4^{|m|}} \big)$, $m \in  \mathbb{Z}$.
Set $k_m^+:= M 4^{|m|} \frac{ \Psi_{s,\eta}^{+}(m\eta)
-\Psi_{s,\eta}^{+}((m-1)\eta)} { \varepsilon }$, $m \in \mathbb{Z}$. 
For the step function $\Psi_{s,\eta}^{+}$, 
in the neighborhood of every possible discontinuous point $m \eta$, $m \in \mathbb{Z}$, 
if $\Psi_{s,\eta}^{+}(m\eta) \geq \Psi_{s,\eta}^{+}((m-1)\eta)$,  
then for any $y \in I_m\cup I_{m+1}$, $m \in \mathbb{Z}$, we define 
\begin{equation*}
\Psi_{s,\eta,\varepsilon}^{+}(y)= 
\begin{cases}
\Psi_{s,\eta}^{+}((m-1)\eta),  
& y \in I_m \setminus I_m^-   \\
\Psi_{s,\eta}^{+}((m-1)\eta) 
+ k_m^+ \left(y - m\eta + \frac{\varepsilon}{M4^{|m|}} \right), 
& y \in I_{m}^- \\
\Psi_{s,\eta}^{+}(m\eta), 
& y \in I_{m+1}. 
\end{cases}
\end{equation*}
If $\Psi_{s,\eta}^{+}(m\eta) < \Psi_{s,\eta}^{+}((m-1)\eta)$,
then we define
\begin{equation*}
\Psi_{s,\eta,\varepsilon}^{+}(y)= 
\begin{cases}
\Psi_{s,\eta}^{+}((m-1)\eta), 
& y \in I_{m} \\
\Psi_{s,\eta}^{+}((m-1)\eta) 
+ k_m^+ (y - m\eta  ), 
& y \in I_{m}^+  \\
\Psi_{s,\eta}^{+}(m\eta), 
& y \in I_{m+1}\setminus I_{m}^+.
\end{cases}
\end{equation*}
From this construction, the non-negative continuous function $\Psi_{s,\eta,\varepsilon}^{+}$
satisfies  $\Psi^+_{s,\eta}  \leq \Psi_{s,\eta,\varepsilon}^{+}$
and $\int_{\mathbb{R}} [\Psi_{s,\eta,\varepsilon}^{+}(y) - \Psi^+_{s,\eta}(y)] dy < \varepsilon$.
Similarly, for the step function $\Psi_{s,\eta}^{-}$, 
one can construct a non-negative continuous function $\Psi_{s,\eta,\varepsilon}^{-}$
which satisfies $ \Psi_{s,\eta,\varepsilon}^{-} \leq \Psi^-_{s,\eta}$ and 
$\int_{\mathbb{R}} [\Psi^-_{s,\eta}(y) - \Psi_{s,\eta,\varepsilon}^{-}(y)] dy < \varepsilon$.
Consequently, in view of \eqref{DiretRiemStep 01}, we obtain that,  for $\eta$ small enough,
\begin{align}\label{estimate approximation}
\int_{\mathbb{R}}|\Psi_{s,\eta,\varepsilon}^{+}(y)- \Psi_{s,\eta,\varepsilon}^{-}(y)|dy< 3\varepsilon.
\end{align}
For brevity, set $c_{s,l,n}=\sqrt{2\pi n}~\sigma_s e^{n \Lambda^*(q+l)}$
and $T_{n,l}^x = \log |G_n x|-n(q+l)$.
Recalling that $\Psi_s(y)=e^{-sy}\psi(y)$, we write 
\begin{align} \label{mainresult3 estimate}
&\  \left|   
c_{s,l,n} \mathbb{E}\left[\varphi(X_{n}^{x}) \psi(T_{n,l}^x ) \right]
 - \bar r_{s}(x) \nu_s(\varphi) \int_{\mathbb{R}}\Psi_s(y)dy\right| \nonumber\\
\leq  & \
\left| 
c_{s,l,n} \mathbb{E}\left\{\varphi(X_{n}^{x})e^{s T_{n,l}^x }
\left[\Psi_s(T_{n,l}^x )-\Psi_{s,\eta,\varepsilon}^{+}(T_{n,l}^x )\right]\right\}
\right|  \nonumber\\  
& \ +
\left| c_{s,l,n} 
\mathbb{E}
\left[\varphi(X_{n}^{x})e^{s T_{n,l}^x }\Psi_{s,\eta,\varepsilon}^{+}(T_{n,l}^x )\right]
 - \bar r_{s}(x) \nu_s(\varphi)  \int_{\mathbb{R}}\Psi_{s,\eta,\varepsilon}^{+}(y)dy\right|  \nonumber\\
&\  +
\left|{r_{s}(x)}\pi_{s}(\varphi r_{s}^{-1})\int_{\mathbb{R}}\Psi_{s,\eta,\varepsilon}^{+}(y)dy
 - \bar r_{s}(x) \nu_s(\varphi)  \int_{\mathbb{R}}\Psi_s(y)dy\right|  \nonumber\\
= & \   J_1 + J_2 + J_3. 
\end{align}
To control $J_2$, we shall verify the conditions of Theorem \ref{main theorem2}.
Noting that the function $y\mapsto e^{sy}\Psi_{s,\eta,\varepsilon}^{+}(y)$ is non-negative and continuous,
it remains to check the condition \eqref{conditiong002}.
By the construction of $\Psi_{s,\eta,\varepsilon}^{+}$
one can verify that there exists a constant 
$\varepsilon_1 \in (0, \min\{M \eta, \eta/3\})$ such that 
\begin{align} \label{Upper Bound Direct Riem}
\int_{\mathbb{R}}\sup_{y'\in\mathbb{B}_{\varepsilon_1}(y)} \Psi_{s,\eta,\varepsilon}^{+}(y') dy
\leq & \  2 \eta \sum_{m\in \mathbb{Z}} \sup_{y\in [m\eta, (m+1)\eta)}\Psi_{s,\eta}^{+}(y)  \nonumber\\
= & \  2 \eta \sum_{m\in \mathbb{Z}} \sup_{y\in [m\eta, (m+1)\eta)}\Psi_{s}(y) <+\infty, 
\end{align}
where the series is finite since the function $\Psi_s$ is directly Riemann integrable.
Hence, applying Theorem \ref{main theorem2} to 
$y\mapsto e^{sy}\Psi_{s,\eta,\varepsilon}^{+}(y)$,
we get
\begin{align} \label{mainresult3 estimate I2}
\lim_{n\to \infty}
\sup_{x\in \mathcal{S}}\sup_{|l|\leq l_n }
J_2 = 0. 
\end{align}
For $J_3(x)$, recall that $\Psi_{s,\eta,\varepsilon}^{-}
\leq \Psi_s \leq \Psi_{s,\eta,\varepsilon}^{+}$. 
Using \eqref{estimate approximation} and the fact that $r_s$ is uniformly bounded on $\mathcal{S}$,
we get that there exists a constant $C_s>0$ such that
\begin{align}  \label{mainresult3 estimate I3}
\sup_{x \in \mathcal{S}} J_3 \leq C_s \varepsilon.  
\end{align}
For $J_1$, 
note that 
$e^{sy}\Psi_{s,\eta,\varepsilon}^{-}(y)
\leq e^{sy}\Psi_s(y)\leq e^{sy}\Psi_{s,\eta,\varepsilon}^{+}(y)$, $y\in\mathbb{R}$. 
Combining this with the positivity of $\varphi$, it holds that 
\begin{align}
|J_1 &|  
\leq 
\left| c_{s,l,n}
\mathbb{E}
\left\{\varphi(X_{n}^{x})e^{s T_{n,l}^x}
\left[\Psi_{s,\eta,\varepsilon}^{+}(T_{n,l}^x)
-\Psi_{s,\eta,\varepsilon}^{-}(T_{n,l}^x)\right]\right\}\right|  \nonumber\\
\leq  & \ 
\left| c_{s,l,n}  
\mathbb{E}
\left[\varphi(X_{n}^{x})e^{s T_{n,l}^x}\Psi_{s,\eta,\varepsilon}^{+}(T_{n,l}^x)\right]
 - \bar r_{s}(x) \nu_s(\varphi) \int_{\mathbb{R}}\Psi_{s,\eta,\varepsilon}^{+}(y)dy\right|  \nonumber\\
& \  +
\left| c_{s,l,n}  
\mathbb{E}
\left[\varphi(X_{n}^{x})e^{s T_{n,l}^x}\Psi_{s,\eta,\varepsilon}^{-}(T_{n,l}^x)\right]
 - \bar r_{s}(x) \nu_s(\varphi) \int_{\mathbb{R}}\Psi_{s,\eta,\varepsilon}^{-}(y)dy\right|  \nonumber\\
& \ +
\left| \bar r_{s}(x) \nu_s(\varphi)  \int_{\mathbb{R}}\Psi_{s,\eta,\varepsilon}^{+}(y)dy
 - \bar r_{s}(x) \nu_s(\varphi)  \int_{\mathbb{R}}\Psi_{s,\eta,\varepsilon}^{-}(y)dy\right|  \nonumber\\
= & \  J_{11} + J_{12} + J_{13}. \nonumber
\end{align}
Using \eqref{mainresult3 estimate I2}, it holds that, as $n \to \infty$, 
$J_{11} \to 0$, uniformly in $x \in \mathcal{S}$ and $|l| \leq l_n$. 
For $J_{12}$, 
note that the function $y\mapsto e^{sy}\Psi_{s,\eta,\varepsilon}^{-}(y)$ is non-negative and continuous.
By the construction of $\Psi_{s,\eta,\varepsilon}^{-}$, 
similarly to \eqref{Upper Bound Direct Riem}, one can verify that 
there exists $\varepsilon_2>0$ such that 
$
\int_{\mathbb{R}}\sup_{y'\in\mathbb{B}_{\varepsilon_2}(y)} \Psi_{s,\eta,\varepsilon}^{-}(y') dy<+\infty.
$
We deduce from Theorem \ref{main theorem2} that 
$J_{12} \to 0$  as $n \to \infty$, 
uniformly in $x \in \mathcal{S}$ and $|l| \leq l_n$. 
For $J_{13}$, we use \eqref{estimate approximation} to get that
$J_{13} \leq C_s \varepsilon$. 
Consequently, we obtain that, as $n \to \infty$,   
$J_1 \leq C_s \varepsilon$, 
uniformly in $x \in \mathcal{S}$ and $|l| \leq l_n$. 
This, together with  \eqref{mainresult3 estimate}, 
\eqref{mainresult3 estimate I2}-\eqref{mainresult3 estimate I3}, implies that
\begin{align*}
\lim_{n\to \infty} \sup_{x\in \mathcal{S}} \sup_{|l|\leq l_n } 
\Big|   
c_{s,l,n} \mathbb{E}\left[\varphi(X_{n}^{x}) \psi(T_{n,l}^x ) \right]
 - \bar r_{s}(x) \nu_s(\varphi)  \int_{\mathbb{R}}\Psi_s(y)dy
 \Big|
  \leq  C_s \varepsilon.
\end{align*}
Since $\varepsilon>0$ is arbitrary, we conclude the proof of Theorem \ref{main theorem3}. 
\end{proof}

\begin{proof}[\textit{Proof of Theorem \ref{Thm-Neg-s-Target}}]
Following the proof of Theorem \ref{main theorem2}, one can verify that the asymptotic \eqref{Asy-s-Posi}
holds true for $s<0$ small enough and for $\psi$ satisfying condition \eqref{condition g}. 
The passage to a directly Riemann integrable function $\psi$ can be done by using the same approximation
techniques as in the proof of Theorem \ref{main theorem3}. 
\end{proof}

\section{Proof of Theorems \ref{Thm-LDP-Norm}, \ref{Thm-LDP-Norm-Negs} and \ref{Theorem local LD001}}

\begin{proof}[\textit{Proof of Theorems \ref{Thm-LDP-Norm} and \ref{Thm-LDP-Norm-Negs}}]
We first give a proof of  Theorem \ref{Thm-LDP-Norm}. 
Since $\log |G_nx|\leq\log \|G_n\|$ 
and the function $\bar r_s$ is strictly positive and uniformly bounded on $\mathcal{S}$,
applying Theorem \ref{main theorem1} we get the lower bound: 
\begin{align} \label{NormLow a}
\liminf_{n\to\infty}
\inf_{ |l| \leq l_n}  \frac{1}{n} \log \mathbb{P}(\log\|G_n\| \geq n(q+l) )
\geq - \Lambda^*(q). 
\end{align}
For the upper bound, since all matrix norms are equivalent, 
there exists a positive constant $C$ which does not depend on the product $G_n$ such that 
$\log \| G_n \| \leq \max_{ 1\leq i \leq d} \log |G_n e_i| + C,$
where $(e_i)_{1 \leq i \leq d}$ is the canonical orthonormal basis in $\mathbb{R}^d$. 
From this inequality, we deduce that
\begin{align*}
\mathbb{P}(\log\|G_n\| \geq n(q+l) ) 
 \leq  \sum_{i=1}^d \mathbb{P} \Big( \log |G_n e_i| \geq n \big( q+l - C/n \big)  \Big). 
\end{align*}
Using Lemma \ref{lemmaCR001}, we see that
there exists a constant $C_s>0$ such that 
$e^{n [\Lambda^*(q+l-C/n) - \Lambda^*(q+l)]} \leq C_s$, uniformly in $|l| \leq l_n$ and $n \geq 1$. 
Again by Theorem \ref{main theorem1}, 
we obtain the upper bound: 
\begin{align*}
\limsup_{n\to\infty}
\sup_{|l|\leq l_n} 
\frac{1}{n} \log \mathbb{P}(\log\|G_n\| \geq n(q+l) )
\leq - \Lambda^*(q). 
\end{align*}
This, together with \eqref{NormLow a}, proves Theorem \ref{Thm-LDP-Norm}. 
Using Theorem \ref{Thm-Neg-s}, 
the proof of Theorem \ref{Thm-LDP-Norm-Negs} can be carried out in the same way.  
\end{proof}


\begin{proof}[\textit{Proof of Theorem \ref{Theorem local LD001}}]
Without loss of generality, we assume that the function $\varphi$ is non-negative.  
From Theorem \ref{main theorem3}, 
we deduce that there exists a
sequence $(r_n)_{n\geq1}$, determined by the matrix law $\mu$
such that $r_n\to 0$ as $n \to \infty$ and, 
uniformly in  $x \in \mathcal{S}$, $|l| \leq l_n$ and
$0 \leq \Delta \leq o(n)$, 
it holds that
\begin{align}\label{LDDelta001}
&\mathbb{E} \Big[ \varphi(X_{n}^{x}) \mathbbm{1}_{ \{ \log|G_n x|  \geq n(q+l) + a + \Delta \} } \Big] \nonumber \\
&\qquad \qquad \qquad = \frac{ \bar r_{s}(x) }{s\sigma_s\sqrt{2\pi n}}
 e^{ -n \Lambda^*(q + l + \frac{a + \Delta}{n}) } \Big[  \nu_s(\varphi) +  r_n  \Big].
\end{align}
Taking the difference of \eqref{LDDelta001} with $\Delta=0$  and with $\Delta>0$,  we get, as $n \to \infty$, 
\begin{align*}
&\mathbb{E} \Big[ \varphi(X_{n}^{x}) 
   \mathbbm{1}_{ \{ \log|G_n x| \in n(q+l) + [a,a+\Delta)  \} } \Big]  \nonumber \\
&\qquad \qquad \qquad = I_{\Delta}(n)  \frac{ \bar r_{s}(x) }{s\sigma_s\sqrt{2\pi n}} e^{ -n \Lambda^*(q + l) }
   \Big[  \nu_s(\varphi) +  r_n  \Big],
\end{align*}
where
\begin{align*}
I_{\Delta}(n): = 
 e^{n \Lambda^*(q + l) - n \Lambda^*(q + l + \frac{a}{n}) }  
- e^{n \Lambda^*(q + l) - n \Lambda^*(q + l+ \frac{ a + \Delta }{n} )}. 
\end{align*}
An elementary analysis using Lemma \ref{lemmaCR001} shows that
$$I_{\Delta}(n) \sim e^{-sa}(1 - e^{-s\Delta}),$$ 
uniformly in $|l| \leq l_n$ and $\Delta_n \leq \Delta \leq o(n)$,
for any $(\Delta_n)_{n\geq1}$ converging to $0$ 
slowly enough ($\Delta_n^{-1} = o( r_n^{-1} )$).
%
%
%
%
This concludes the proof of Theorem \ref{Theorem local LD001}.
\end{proof}


\begin{thebibliography}{2}


\bibitem{BR1960} Bahadur R. R., Rao R. R.: On deviations of the sample mean. \emph{The Annals of Mathematical Statistics}, 31(4), 1015-1027, 1960.


\bibitem{BQ2016} Benoist Y., Quint J. F.: Central limit theorem for linear groups. 
\emph{The Annals of Probability}, 44(2), 1308-1340, 2016.

\bibitem{BQ2017} Benoist Y., Quint J. F.: Random walks on reductive groups. \emph{Springer International Publishing}, 2016.

\bibitem{Borovkov2008} Borovkov A. A., Borovkov K. A.: Asymptotic analysis of random walks.  \emph{Cambridge University Press}, 2008.

\bibitem{Bougerol1985} Bougerol P., Lacroix J.: Products of random matrices with applications to Schr\"{o}dinger operators. Birkh\"{a}user Boston, 1985.

\bibitem{Bre2005} Breuillard E.: Distributions diophantiennes et th\'{e}or\`{e}me limite local sur $\mathbb{R}^d$. \emph{Probability Theory and Related Fields},  132(1): 13-38, 2005.


\bibitem{BDGM2014} Buraczewski D., Damek E., Guivarc'h Y.,  Mentemeier S.: On multidimensional Mandelbrot cascades. \emph{Journal of Difference Equations and Applications}, 20(11), 1523-1567, 2014.

\bibitem{BS2016} Buraczewski D., Mentemeier S.: Precise large deviation results for products of random matrices. \emph{Annales de l'Institut Henri Poincar\'{e}, Probabilit\'{e}s et Statistiques}. Vol. 52, No. 3, 1474-1513, 2016.

\bibitem{BurCollDamZien2016}  
Buraczewski D., Collamore J., Damek E., Zienkiewicz J.:
Large deviation estimates for exceedance times of perpetuity sequences and their dual processes.
\emph{The Annals of Probability} 44(6), 3688-3739, 2016. 

%

\bibitem{Daniels1954} Daniels H. E.: Saddlepoint approximations in statistics. \emph{The Annals of Mathematical Statistics}, 631-650, 1954.

\bibitem{Dembo2009} Dembo A., Zeitouni O.: Large deviations techniques and applications. \emph{Springer Science and Business Media}, 2009.

\bibitem{Fedoryuk1987} Fedoryuk M. V.:  Asymptotic, Integrals and Series, Nauka, 1987 (in Russian).

\bibitem{Furman2002} Furman A.: Random walks on groups and random transformations. \emph{Handbook of dynamical systems}, 1, 931-1014, 2002.

\bibitem{Furstenberg1963} Furstenberg H., Noncommuting random products. \emph{Transactions of the American Mathematical Society}, 108(3), 377-428, 1963.

\bibitem{FursKesten1960} Furstenberg H., Kesten H.: Products of random matrices. \emph{The Annals of Mathematical Statistics}, 31(2), 457-469, 1960.

\bibitem{gnedenko1948}
Gnedenko B.~V.:
On a local limit theorem of the theory of probability.
\emph{Uspekhi Matematicheskikh Nauk}, 3(3):187-194, 1948.

\bibitem{GG1996} Goldsheid I. Y., Guivarc'h Y.: Zariski closure and the dimension of the Gaussian law of the product of random matrices. \emph{Probability Theory and Related Fields}, 105(1), 109-142, 1996.

\bibitem{GLE2017} Grama I., Lauvergnat R., Le Page \'{E}.: Conditioned local limit theorems 
for random walks defined on finite Markov chains. \emph{arXiv preprint} arXiv:1707.06129, 2017.

\bibitem{GE2017} Grama I., Le Page \'{E}.: Bounds in the local limit theorem for a random walk conditioned to stay positive. \emph{In International Conference on Modern Problems of Stochastic Analysis and Statistics.} 103-127, Springer 2017.




\bibitem{Guivarch2015} Guivarc'h Y.: Spectral gap properties and limit theorems for some random walks and dynamical systems. \emph{Proc. Sympos. Pure Math}. 89, 279-310, 2015.

	
\bibitem{GE2016} Guivarc'h Y., Le Page \'{E}.: Spectral gap properties for linear random walks and Pareto's asymptotics for affine stochastic recursions. \emph{Annales de l'Institut Henri Poincar\'{e}, Probabilit\'{e}s et Statistiques}. Vol. 52. No. 2, 503-574, 2016.

\bibitem{GR1985} Guivarc'h Y., Raugi A.: Frontiere de Furstenberg, propri\'{e}t\'{e}s de contraction et th\'{e}or\`{e}mes de convergence. \emph{Probability Theory and Related Fields}, 69(2): 187-242, 1985.

\bibitem{GU2005} Guivarc'h Y., Urban R.: Semigroup actions on tori and stationary measures on projective spaces. \emph{Studia Math}. 171, no. 1, 33-66, 2005.

\bibitem{Hennion1997} Hennion H.: Limit theorems for products of positive random matrices. 
\emph{The Annals of Probability}, 25(4): 1545-1587, 1997.

\bibitem{HH01} Hennion H., Herv\'{e} L.: Limit theorems for Markov chains and stochastic properties of dynamical systems by quasi-compactness. Vol. 1766, \emph{Lecture Notes in Mathematics}. Springer-Verlag, Berlin, 2001.

\bibitem{Hen-Herve2004} Hennion H., Herv\'{e} L.: Central limit theorems for iterated random Lipschitz mappings. 
 \emph{The Annals of Probability}, 32: 1934-1984, 2004.
 
\bibitem{IbrLinnik65} Ibragimov I.A., Linnik Yu.V.: Independent and stationary sequences of random variables. Wolters, Noordhoff Pub., 1975.




\bibitem{Kesten1973} Kesten H.: Random difference equations and renewal theory for products of random matrices. \emph{Acta Mathematica}, vol. 131(1): 207-248, 1973.

\bibitem{Kingman1973} Kingman J. F. C.: Subadditive ergodic theory. \emph{The Annals of Probability}, 883-899, 1973.

\bibitem{LePage1982} Le Page \'{E}.: Th\'{e}or\`{e}mes limites pour les produits de matrices al\'{e}atoires. \emph{In Probability measures on groups}. Springer Berlin Heidelberg, 258-303, 1982.

\bibitem{Petrov65} Petrov V. V.: On the probabilities of large deviations for sums of independent random variables. 
 \emph{Theory of Probability and its Applications}, 10(2): 287-298, 1965. 

\bibitem{Petrov75book} Petrov V. V.: Sums of independent random variables. \emph{Springer}, 1975. 



\bibitem{Richter1957} Richter W.: Local limit theorems for large deviations. 
\emph{Theory of Probability and its Applications}, 2(2): 206-220, 1957.


\bibitem{Sheep1964} Sheep L. A.:  A local limit theorem.
\emph{The Annals of Mathematical Statistics}, 35: 419-423, 1964.

\bibitem{Stone1965} Stone C.: A local limit theorem for nonlattice multi-dimensional distribution functions.
\emph{The Annals of Mathematical Statistics}, 36(2): 546-551, 1965.




\bibitem{XGL19} Xiao H., Grama I., Liu Q.: 
Berry-Esseen bound and precise moderate deviations for products of random matrices, \emph{Submitted}, 2019.
    

\end{thebibliography}
\end{document}